\documentclass{birkjour}
 \newtheorem{thm}{Theorem}[section]
 
 \newtheorem{lem}[thm]{Lemma}
 
 \newtheorem*{nonumtheorem}{Theorem}

 \theoremstyle{definition}
 \newtheorem{defn}[thm]{Definition}
 \theoremstyle{remark}
 \newtheorem{rem}[thm]{Remark}
 \newtheorem*{ex}{Example}

 \usepackage{amsmath}
\usepackage{amsfonts}
\usepackage{amssymb}
 \usepackage[mathscr]{eucal}
 \usepackage[cp1251]{inputenc}
\usepackage[OT1,T1,TS1,T2A]{fontenc}
\usepackage[russian,english]{babel}
\numberwithin{equation}{section}
\newcommand{\ccomma}{\mathpunct{\raisebox{0.5ex}{,}}}

\begin{document}


\title[Self-adjoint boundary conditions]
 {Self-adjoint boundary conditions \\ for
the prolate spheroid differential operator.}

\author[Victor Katsnelson]{Victor Katsnelson}

\address{%
Department of Mathematics, Weizmann Institute, Rehovot, 7610001, Israel}

\email{victor.katsnelson@weizmann.ac.il; victorkatsnelson@gmail.com}




\date{January 1, 2016}
\dedicatory{Dedicated to Heinz Langer in the occasion of his 80th birthday}

\begin{abstract}
We consider the formal prolate spheroid differential operator on a finite symmetric interval and describe all its self-adjoint boundary conditions. Only one of these boundary conditions corresponds to a self-adjoint differential operator which commute with the Fourier operator truncated on the considered finite symmetric interval.
\end{abstract}

\maketitle
Hereinafter, \(\langle\,.\,,\,.\,\rangle\) means the standard scalar product in the Hilbert space \(L^2([-a,a])\):
\begin{equation*}
\langle u,v\rangle=\int\limits_{[-a,a]}u(t)\overline{v(t)}\,dt, \ \ \forall\,u,v\in L^2([-a,a]).
\end{equation*}
\section{Introduction.}
The study of the spectral theory of the Fourier operator
restricted on a finite symmetric interval \([-a,a]\):
\begin{multline}
\label{TrFSI}%
 \hspace{4.0ex}
(\mathscr{F}_{E}x)(t)=\frac{1}{\sqrt{2\pi}}\int\limits_{-a}^{a}e^{it\xi}x(\xi)\,d\xi,\quad
\ \ t\in{}E,\ \ E=[-a,a],\\
\mathscr{F}_{E}:\,\,L^2(E)\to{}L^2(E)\,,\hspace{4.0ex}
\end{multline}
is closely related to study of the differential operator generated
by the differential expression (or \emph{formal differential operator}) \(L\):
\begin{equation}
\label{FDOFI}%
(Lx)(t)=
-\frac{d\,\,}{dt}\bigg(\bigg(1-\frac{t^2}{a^2}\bigg)\frac{dx(t)}{dt}\bigg)+t^2x(t)\,.
\end{equation}
The relationship between the spectral theory of the integral operator
\(\mathscr{F}_{E}^{\ast}\mathscr{F}_{E}\), \(E=[-a,a]\), and a
differential operator generated by the differential expression \(L\)
was discovered in the series of remarkable papers \cite{SlPo},
\cite{LaP1}, \cite{LaP2}, where this relationship has been
ingeniously used
 for developing the spectral theory of the operator
\(\mathscr{F}_{E}^{\ast}\mathscr{F}_{E}\). (See also \cite{Sl2},
\cite{Sl3}.) Actually the reasoning of  \cite{SlPo}, \cite{LaP1},
\cite{LaP2} can be easily applied to the spectral theory of the
operator \(\mathscr{F}_{E}\) itself.

It is a certain system of
eigenfunctions related to the differential expression \(L\) which was used in \cite{SlPo},
\cite{LaP1}, \cite{LaP2}. These eigenfunctions are known as \emph{prolate spheroidal
wave functions}. The prolate spheroidal wave functions themselves were used
much before the series of the papers \cite{SlPo}, \cite{LaP1},
\cite{LaP2} was published. These functions naturally appear by
separation of variables for the Laplace equation in spheroidal
coordinates. However it was the works \cite{SlPo}, \cite{LaP1},
\cite{LaP2} where the prolate functions were first used for
solving the spectral problem related to the Fourier analysis on
a finite symmetric interval. Until now, there is no clear
understanding why the approach used in \cite{SlPo}, \cite{LaP1},
\cite{LaP2} works. This is a \emph{lucky accident} which still waits
for its explanation. (See \cite{Sl3}.)

Actually eigenfunctions are related not to the the differential
expression itself but to a certain differential operator. This differential operator is
generated not only by the differential expression but also by
certain boundary conditions. In the case \(E=(-\infty,\infty)\),
the differential operator generated by the differential expression
\(-\frac{d^2\,}{dt^2}+t^2\) on the class smooth finite functions
(or the class of smooth fast decaying functions) is essentially
self-adjoint: the closure of this operator is a self-adjoint
operator. Thus in the case \(E=(-\infty,\infty)\) there is no need
to discuss the boundary condition.

In contrast to the case \(E=(-\infty,\infty)\), in the case
\(E=[-a,a],\,\,0<a<\infty\), the minimal differential operator related to the formal differential operator
\(\displaystyle-\frac{d\,\,}{dt}\Big(1-\frac{t^2}{a^2}
\frac{d\,\,}{dt}\bigg)+t^2\)
is symmetric but is \emph{not} self-adjoint. This minimal operator
admits the family of self-adjoint extensions. Each of this
self-adjoint extensions is described by a certain boundary
conditions at the end points of the interval \([-a,a]\). The set
of all such extensions can be parameterized by the set of all
\(2\times2\) unitary matrices.

 It turns out that only one of these extensions
commutes with the truncated Fourier operator
\(\mathscr{F}_{E},\,\,E=[-a,a]\). To our best knowledge, until now
no attention was paid to this aspect. In the present paper, we  investigate the question which extensions of the
minimal differential operator generated by
\(L\), \eqref{FDOFI}, commute with \(\mathscr{F}_{E}\), \(E=[-a,a]\).

The formal operator \(L\) is of the form
\begin{subequations}
\label{LDO}
\begin{gather}
\label{LDO1}
L=M+Q,\\
\intertext{where}
\label{LDO2}
(Mx)(t)=-\frac{d\,\,}{dt}\bigg(\bigg(1-\frac{t^2}{a^2}\bigg)\frac{dx(t)}{dt}\bigg),\\[1.0ex]
\label{LDO3}
(Qx)(t)=t^2x(t).
\end{gather}
\end{subequations}
The operator \(L\) is said to be \emph{the prolate spheroid differential operator}.\\
The operator \(M\) is said to be \emph{the Legendre differential operator}.\\

The operator \(Q\) is a bounded self-adjoint operator in \(L^2([-a,a])\). Therefore the operators \(L\) and \(M\) are "equivalent" from the viewpoint of
the extension theory:
if one of these operators is self-adjoint on some domain of definition \(\mathcal{D}\),
then the other is self-adjoint on \(\mathcal{D}\) as well.

\section{Analysis of solutions of the equation
\(\boldsymbol{Mx=\lambda{}x}\) near singular points.}
For the differential equation
\begin{equation}
\label{DEFEWP}
-\frac{d\,\,}{dt}\bigg(\Big(1-\frac{t^2}{a^2}\Big)\frac{dx(t)}{dt}\bigg)
=\lambda{}x(t),\quad
 t\in\mathbb{C},
\end{equation}
considered in complex plane, the points \(-a\) and \(a\) are the
regular singular point. Let us investigate the asymptotic behavior
of solutions of the equation \eqref{DEFEWP} near these points. (Actually we need
to know this behavior only for real \(t\in(-a,a)\) only, but it is
much easier to investigate this question using some knowledge from
the analytic theory of differential equation.) Concerning the
analytic theory of differential equation see \cite[Chapter 5]{Sm}.

Let us outline an analysis of solution of the equation near the
point \(t=-a\). Change of variable
\[t=-a+s,\ \ \  x(-a+s)=y(s)\]
reduces the equation \eqref{DEFEWP} to the form
\begin{equation}
\label{ChVar}%
 s\frac{d^2y(s)}{ds^2}+f(s)\frac{dy(s)}{ds}+g(s,\lambda)y(s)=0\,,
\end{equation}
where \(f(s)\) and \(g(s)\) are functions holomorphic within the
disc \(|s|<2a\), moreover \(f(0)=1\):
\begin{equation}
\label{psex}%
f(s)=1+\sum\limits_{k=1}^{\infty}f_ks^k,\quad
g(s,\lambda)=\sum\limits_{k=0}^{\infty}g_k(\lambda)s^k\,.
\end{equation}
An explicit calculation with power series gives:
\begin{equation}
\label{ExEx1} f_1=-\frac{1}{2a};\ \ \
g_0=\frac{\lambda{}a}{2},\ \ g_1=
\frac{\lambda}{4}\,.
\end{equation}
 Now we turn to the analytic theory of differential equations.
 The results of this theory which we need are
 presented for example in \cite[Chapter 5]{Sm}, see especially
section \textbf{98} there. We seek the solution of the equation
\eqref{ChVar}-\eqref{psex} in the form
\[y(s)=s^{\rho}\sum\limits_{k=0}^{\infty}c_ks^k\,.\]
Substituting this expression to the left-hand side of the equation
\eqref{ChVar}-\eqref{psex} and equating the coefficients,  we obtain the equations for the
determination of \(\rho\) and \(c_k\). In particular, the equation
corresponding to the power \(s^{\rho-1}\) is of the form:
\[c_0\,\rho^2=0\,.\]
The coefficient \(c_0\) plays the role of a normalizing  constant,
and we may take
\begin{equation}
\label{NoCo}%
 c_0=1\,.
\end{equation}
Equation for \(\rho\), the so called \emph{characteristic
equation}, is of the form
\begin{equation}
\label{ChaEq}%
 \rho^2=0.
\end{equation}
This equation has the root \(\rho=0\) and this root is of multiplicity two.
According to general theory, the equation
\eqref{ChVar}-\eqref{psex} has two solutions \(y_1(s)\) and
\(y_2(s)\) possessing the properties:

The solution \(y_1(s)\) is a function holomorphic is the disc
\(|s|<2a\) satisfying the normalizing condition \(y_1(0)=1\). The
solution \(y_2(s)\) is of the form \(y_2(s)=y_1(s)\,\ln{}s+z(s)\),
where \(z(s)\) is a function holomorphic  in the disc \(|s|<2a\)
and satisfying the condition \(z(0)=0\). Returning to the variable
\(t=-a+s\), we get the following result:
\begin{lem}%
\label{ABSNS}%
Let \(M\) be the differential expression defined by \eqref{LDO2},
and \(\lambda\in\mathbb{C}\) be arbitrary fixed.
\begin{enumerate}
\item
There exist two solutions \(x_{1}^{-}(t,\lambda)\) and
\(x_{2}^{-}(t,\lambda)\) of
the equation %
 \begin{math}%
Mx(t)=\lambda{}x(t)
\end{math} %
possessing the properties:
\begin{enumerate}
\item
The
function \(x_{1}^{-}(t,\lambda)\) is holomorphic in the disc
\(|t+a|<2a\), and satisfies the normalizing condition
\(x_{1}^{-}(-a,\lambda)=1\)\,;
\item
 The function
\(x_{2}^{-}(t,\lambda)\) is of the form %
\[x_{2}^{-}(t,\lambda)=x_{1}^{-}(t,\lambda)\,\ln{}(t+a)+w^{-}(t,\lambda),\]
where the function \(w^{-}(t,\lambda)\) is holomorphic in the disc
\(|t+a|<2a\) and satisfies the condition \(w^{-}(-a,\lambda)=0\)\,.
\end{enumerate}
\item
There exist two solutions \(x_{1}^{+}(t,\lambda)\) and
\(x_{2}^{+}(t,\lambda)\) of the equation
\begin{math}%
Mx(t)=\lambda{}x(t)
\end{math}
possessing the properties:
\begin{enumerate}
\item
 The
function \(x_{1}^{+}(t,\lambda)\) is holomorphic in the disc
\(|t-a|<2a\), and satisfy the normalizing condition
\(x_{1}^{+}(a,\lambda)=1\)\,;
\item
 The function
\(x_{2}^{+}(t,\lambda)\) is of the form %
\[x_{2}^{+}(t,\lambda)=x_{1}^{+}(t,\lambda)\,\ln{}(a-t)+w^{+}(t,\lambda),\]
where the function \(w^{+}(t,\lambda)\) is holomorphic in the disc
\(|t+a|<2a\) and satisfy the condition \(w^{+}(a,\lambda)=0\)\,.
\end{enumerate}
\end{enumerate}
\end{lem}%

For a fixed \(\lambda\), the solutions \(x_{1}^{-}(t,\lambda)\).
\(x_{2}^{-}(t,\lambda)\) are linearly independent. Therefore
arbitrary solution \(x(t,\lambda)\) of the equation \eqref{DEFEWP}
can be expanded into a linear combination
\begin{subequations}
\label{LCom}
\begin{equation}
\label{LComM}%
x(t,\lambda)=c_1^{-}x_{1}^{-}(t,\lambda)+c_2^{-}x_{2}^{-}(t,\lambda).
\end{equation}
 The solutions \(x_{1}^{+}(t,\lambda)\),
\(x_{2}^{+}(t,\lambda)\) also are linearly independent, and the
solution \(x(t,\lambda)\) can be also expanded into the other
linear combination
\begin{equation}
\label{LComP}%
x(t,\lambda)=c_1^{+}x_{1}^{+}(t,\lambda)+c_2^{+}x_{2}^{+}(t,\lambda).
\end{equation}
\end{subequations}
Here \(c_1^{\pm}\), \(c_2^{\pm}\) are constants (with respect to
\(t\)). The solution \(x_{1}^{-}(t,\lambda)\) is bounded  and the
solution \(x_{2}^{-}(t,\lambda)\) grows logarithmically as
\(t\to{}-a\). Therefore the solution \(x(t,\lambda)\) is square
integrable near the point \(t=-a\). For the same reason, the the
solution \(x(t,\lambda)\) is square integrable near the point
\(t=a\). Thus we prove the following result.
\begin{lem}\label{InArS}
Given \(\lambda\in\mathbb{C}\), then every solution
\(x(t,\lambda)\) of the equation \eqref{DEFEWP} is square integrable:
\begin{equation}
\label{SqInt}%
 \int\limits_{-a}^{a}\big|x(t,\lambda)\big|^2\,dt<\infty\,.
\end{equation}
\end{lem}

 \section{Maximal and minimal differential operators generated by the differential
expression \(\boldsymbol{M}\).}

Various differential operators may be related to the differential expression
\begin{equation}
\label{DiffExp}%
M=-\frac{d\,\,}{dt}\bigg(1-\frac{t^2}{a^2}\bigg)\frac{d\,\,}{dt}\,.
\end{equation}
 Such operators are determined by
boundary conditions which are posed on functions from their domains of
definition.

\begin{defn}
\label{DoDeMaDO}%
 The set \(\mathcal{A}\) is the set of complex-valued
functions \(x(t)\) defined on the open interval \((-a,a)\) and
satisfying the
following conditions:
\begin{enumerate}
\item The derivative
\(\dfrac{dx(t)}{dt}\) of the function \(x(t)\) exists at every
point \(t\) of the interval \((-a,a)\);
\item
The function \(\dfrac{dx(t)}{dt}\) is absolutely continuous on
every compact subinterval of the interval \((-a,a)\);
\end{enumerate}
\end{defn}
\begin{defn}
\label{MaxDO}
The differential operator \(\mathcal{M}_{\textup{max}}\) is defined as follows:
\begin{subequations}
\label{maxdo}
\begin{enumerate}
\item
The domain of definition \(\mathcal{D}_{\mathcal{M}_{\textup{max}}}\)%
of the operator \(\mathcal{M}_{\textup{max}}\)
is:
\begin{equation}%
\label{maxdo1}
\mathcal{D}_{\mathcal{M}_{\textup{max}}}=\lbrace{}x:\,x(t)\in%
L^2((-a,a))\cap{\mathcal{A}}\ \  \textup{and} \ \
 (Mx)(t)\in{}L^2((-a,a))\rbrace,
 \end{equation}%
 where \((Mx)(t)\) is defined\,\footnotemark%
by \eqref{LDO2}.
\item
The action of the operator  \(\mathcal{M}_{\textup{max}}\) is:
\begin{equation}
\label{maxdo2}
\mathcal{M}_{{}_\textup{max}}x=Mx\,,\quad \forall\,x\in\mathcal{D}_{\mathcal{M}_{\textup{max}}}
\end{equation}
\end{enumerate}
\end{subequations}
\footnotetext{Since \(x\in\mathcal{A}\), the expression \((Mx)(t)\) is well defined.}%
The operator \(\mathcal{M}_{\textup{max}}\) is said to be the \emph{maximal differential
operator generated by the differential expression} \(M\).
\end{defn}
\begin{defn}
\label{DoDeMiDO}%
The set \(\mathring{\mathcal{A}}\)
is the set of complex-valued functions \(x(t)\) defined on the open interval
\((-a,a)\) and satisfied the
following conditions:
\begin{enumerate}
\item
The function \((x)(t)\) belongs to the set \(\mathcal{A}\) defined
above;
\item
The support \(\textup{supp}\,x\) of the function \(x(t)\) is a
compact subset of the open interval \((-a,a)\):
\((\textup{supp}\,x)\!\Subset{}(-a,a)\).
\end{enumerate}
\end{defn}
The minimal differential operator \(\mathcal{M}_{{}_\textup{min}}\)
is a restriction of the maximal differential operator \(\mathcal{M}_{{}_\textup{max}}\) on the set
of functions which is some sense vanish at the endpoint of the interval \((-a,a)\).
The precise definition is presented below.
\begin{defn}
\label{MinDO}
\begin{subequations}
\label{mindo}
The operator the \(\mathring{\mathcal{M}}\) is the restriction of the
operator \(\mathcal{M}_{{}_\textup{max}}\) on the set \(\mathring{\mathcal{A}}\)
compactly supported in \((-a,a)\) functions from \(\mathcal{A}\):
\begin{equation}%
\label{mindo1}
\mathcal{D}_{\mathring{\mathcal{M}}}=%
\mathcal{D}_{\!{}_{\mathcal{M}_{\textup{max}}}}\!\cap{}\mathring{\mathcal{A}}\,,\quad %
{\mathring{\mathcal{M}}}\subset\mathcal{M}_{{}_\textup{max}}.
\end{equation}%
The operator \(\mathcal{M}_{{}_\textup{min}}\)
is the closure\,\footnote{%
Since the operator \(\mathring{\mathcal{M}}\) is symmetric and densely defined, it  is closable.}%
 of the operator \(\mathring{\mathcal{M}}\):
\begin{equation}
\label{mindo2}
\mathcal{M}_{{}_\textup{min}}=
\textup{clos}\big(\mathring{\mathcal{M}}\,\big)\,.
\end{equation}
\end{subequations}
The operator \(\mathcal{M}_{\textup{min}}\) is said to be the \emph{minimal differential
operator generated by the differential expression} \(M\).
\end{defn}

\begin{thm}{\ }\\[-2.5ex]
\label{MSt}
\begin{enumerate}
\item
\textit{The operator \(\mathcal{M}_{{}_\textup{min}}\) is symmetric:
\begin{equation*}%
\langle{}\mathcal{M}_{{}_\textup{min}}x,y\rangle=
\langle{}x,\mathcal{M}_{{}_\textup{min}}y{}\rangle\,,
\quad \forall {}x,\,y\in{}%
\mathcal{D}_{\!{}_{\scriptstyle\mathcal{L}_{\textup{min}}}};
\end{equation*}%
In other words, the operator \(\mathcal{M}_{{}_\textup{min}}\) is contained in its adjoint}:
\begin{equation}%
\label{MinSym}%
\mathcal{M}_{{}_\textup{min}}\subseteq(\mathcal{M}_{{}_\textup{min}})^\ast\,;
\end{equation}%
\item
\textit{The operators \(\mathcal{M}_{{}_\textup{min}}\) and \(\mathcal{M}_{{}_\textup{max}}\)
are mutually adjoint}:
\begin{equation}%
\label{ATDO}%
(\mathcal{M}_{{}_\textup{min}})^\ast=\mathcal{M}_{{}_\textup{max}},\quad
(\mathcal{M}_{{}_\textup{max}})^\ast=\mathcal{M}_{{}_\textup{min}}\,.
\end{equation}%
\end{enumerate}
\end{thm}
\begin{proof}
The proof of this theorem can be found in \cite[10.4.7-10.4.11]{HuPy}.
\end{proof}
\section{The boundary linear forms related to the Legendre operator~\(\boldsymbol{M}\).}
\label{SeBLF}
We use the notations
\begin{equation*}%
p(t)=1-\frac{t^2}{a^2}, \quad -a<t<a.
\end{equation*}
In this notation, the formal differential operator \(M\) introduced in
\eqref{DiffExp} is:
\begin{equation*}%
(Mx)(t)=-\frac{d}{dt}\bigg(p(t)\frac{dx(t)}{dt}\bigg), \quad -a<t<a\,.
\end{equation*}
For every \(x,y\in\mathcal{A},\)
\begin{equation*}%
(Mx)(t)\,\overline{y(t)}-x(t)\,{}(\overline{My)(t)}=\frac{d}{dt}[x,y](t),\quad -a<t<a\,,
\end{equation*}%
where
\begin{equation}%
\label{BF}
[x,y](t)=-p(t)\bigg(\frac{dx(t}{dt}\overline{y(t)}-x(t)\overline{\frac{dy(t}{dt}}\bigg)\,\ccomma
\quad -a<t<a,
\end{equation}%
Therefore, for every \(x,y\in\mathcal{A}\) and for every
\(\alpha,\,\beta:\,-a<\alpha<\beta<a\),
\begin{equation}%
\label{GrF}
\int\limits_{\alpha}^{\beta}\Big((Mx)(t)\,%
\overline{y(t)}-x(t)\,{}\overline{(My)(t)}\Big)\,dt=[x,y](\beta)-[x,y](\alpha)\,.
\end{equation}%
\begin{lem}
\label{EBV}
For each \(x,y\in\mathcal{D}_{\mathcal{M}_{\textup{max}}}\), there exist the
limits
\begin{equation}%
\label{LRBFo}%
[x,y]_{-a}\stackrel{\textup{\tiny def}}{=}\lim_{\alpha\to{}-a+0}[x,y](\alpha),\quad
[x,y]^{a}\stackrel{\textup{\tiny def}}{=}\lim_{\beta\to{}a-0}[x,y](\beta)\,,
\end{equation}%
where the expression \([x,y](t)\) is defined in \eqref{BF}.
\end{lem}
\begin{proof}
Since the functions \(x(t),\,y(t), (Mx)(t), (My)(t)\) belong to \(L^2((-a,a))\), then
 \(\int\limits_{-a}^{a}\Big|(Mx)(t)\,%
\overline{y(t)}-x(t)\,{}\overline{(My)(t)}\Big|\,dt<\infty\).
Therefore the limit\\[-2.0ex]
\begin{multline}
\label{PaLi}
\lim_{\substack{\alpha\to{-a+0}\\
\beta\to{\,a-0}}}\int\limits_{\alpha}^{\beta}\Big((Mx)(t)\,%
\overline{y(t)}-x(t)\,{}\overline{(My)(t)}\Big)\,dt=\\[-2.5ex]
\int\limits_{a}^{b}\Big((Mx)(t)\,%
\overline{y(t)}-x(t)\,{}\overline{(My)(t)}\Big)\,dt
\end{multline}
exists.
Comparing \eqref{PaLi} with \eqref{GrF}, we conclude that the limits in \eqref{LRBFo} exist.
\end{proof}{\ }\\[-4.0ex]
Concerning Lemma \ref{EBV} and related results see \cite[10.4.12-10.4.13]{HuPy}.
\begin{lem}
\label{SkLf}
The expressions \([x,y]_{-a}\) and \([x,y]^{a}\),
which were introduced by \eqref{BF} and \eqref{LRBFo},
are well defined for \(x\in\mathcal{D}_{\mathcal{M}_{\textup{max}}}, \
y\in\mathcal{D}_{\mathcal{M}_{\textup{max}}}\). Considered as functions
of \(x,y\in\mathcal{D}_{\mathcal{M}_{\textup{max}}}\), they are
sesquilinear forms. The forms \([x,y]_{-a}\) and \([x,y]^{a}\) are
skew-hermitian:
\begin{equation}
[x,y]_{-a}=-\overline{[y,x]_{-a}}, \quad [x,y]^{a}=-\overline{[y,x]^{a}},
\quad \forall\,x,y\in\mathcal{D}_{\mathcal{M}_{\textup{max}}}.
\end{equation}
\end{lem}
\begin{defn}
\label{DBBF}
The forms \([x,y]_{-a}\) and \([x,y]^{a}\) are said to be \emph{the end point sesquilinear forms related to the differential operator \(M\).}
\end{defn}
\begin{thm}
\label{GreFT}
For every \(x\in\mathcal{D}_{\mathcal{M}_{\textup{max}}},
\,y\in\mathcal{D}_{\mathcal{M}_{\textup{max}}}\), the equality
\begin{equation}
\label{GreF}
\langle{}\mathcal{M}_{\textup{max}}\,x,y\rangle-
\langle{}x,\mathcal{M}_{\textup{max}}\,y\rangle=[x,y]^{a}-[x,y]_{-a}
\end{equation}
holds, where \([x,y]_{-a},\,[x,y]^{a}\) are the  end point forms related to the differential operator \(M\).
\end{thm}
\begin{proof}
The equality \eqref{GreF} is a consequence of the equalities \eqref{PaLi}, \eqref{GrF} and \eqref{LRBFo}.
\end{proof}

\section{The deficiency indices of the operator \(\boldsymbol{\mathcal{M}_{\textup{min}}}\).}
In 1930 John von Neumann, \cite{Neu}, has found a criterion for
the existence of a self-adjoint extension of a symmetric operator
\(A\) and has described all such extensions. This criterion is
formulated in terms of deficiency indices of the symmetric
operator.

\begin{defn}
\label{DefDefInd} Let \(A\) be an operator in a Hilbert space
\(\mathfrak{H}\). We assume that the domain of definition
\(\mathcal{D}_{A}\) is dense in \(\mathfrak{H}\) and that the
operator \(A\) is symmetric, that is\,%
\begin{equation}
\label{SymInc}%
 \langle Ax,y\rangle=\langle x,Ay\rangle,\quad \forall\,x,y\in\mathcal{D}_{A}\,.
\end{equation}
For \(\lambda\in\mathbb{C}\),
consider the orthogonal complement
\begin{equation}
\label{DefSp1}
\mathcal{N}_{\lambda}=\mathfrak{H}\ominus (A-\lambda{}I)\mathcal{D}_{\!A} \,,
\end{equation}
 of  the subspace \((A-\lambda{}I)\mathcal{D}_{\!A}\), or, what is equivalent,
\begin{equation}
\label{DefSp}%
 \mathcal{N}_{\lambda}=\lbrace{}x\in\mathcal{D}_{\!A^{\ast}}:\
 A^\ast{}x=\overline{\lambda}{}x\rbrace\,,
\end{equation}
where \(A^{\ast}\) is the operator adjoint to the operator \(A\),
\(\mathcal{D}_{\!A^{\ast}}\) is the domain of definition of \(A^{\ast}\).

The
subspace \(\mathcal{N}_{\lambda}\) is said to be \emph{the
deficiency subspace of the operator \(A\) corresponding to the value \(\lambda\)}.
\end{defn}
\begin{rem}
\label{FaS}
The equality \eqref{SymInc} implies that \(\mathcal{D}_{A}\subseteq\mathcal{D}_{\!A^{\ast}}\). So the factor space
\(\mathcal{D}_{\!A^{\ast}}/\mathcal{D}_{A}\) is defined.
\end{rem}
\begin{nonumtheorem}[von Neumann]
Let \(A\) be an operator in the Hilbert space \(\mathfrak{H}\). We assume that the domain of definition
\(\mathcal{D}_{A}\) is dense in \(\mathfrak{H}\) and that the
operator \(A\) is symmetric. Then
\begin{enumerate}
\item
 The dimension
\(\dim{}\mathcal{N}_{\lambda}\) is constant in the upper
half-plane and in the lower half-plane:
\begin{subequations}
\label{DefInd}
\begin{align}
\label{DefInd+}
\dim{}\mathcal{N}_{\lambda}= n_{+}\,,\quad & \forall\,\lambda:\,\textup{Im}\,\lambda>0,\\
\label{DefInd-} \dim{}\mathcal{N}_{\lambda}= n_{-}\,,\quad &  \forall\,\lambda:\,
\textup{Im}\,\lambda<0\,,
\end{align}
each of \(n_{+},n_{-}\) may be either non-negative integer or \(+\infty\).
\end{subequations}
The numbers \(n^+\) and \(n^{-}\) are said to be the
\emph{deficiency indices of the operator~\(A\)}.
\item
For the dimension  of
the factor space \(\mathcal{D}_{\!A^{\ast}}/\mathcal{D}_{A}\) the equality
\begin{equation}
\label{DiFa}
\textup{dim}(\mathcal{D}_{\!A^{\ast}}/\mathcal{D}_{A})=n_{+}+n_{-}
\end{equation}
holds.
\end{enumerate}
\end{nonumtheorem}
\begin{nonumtheorem}[von Neumann] Let \(A\) be a densely defined symmetric operator
and \(n_{+},\,n_{-}\) are its deficiency indices. Then
\begin{enumerate}
\item
The operator
\(A\) is self-adjoint if and only if \(n_+=n_-=0\).
\item
The  operator \(A\) admits self-adjoint
extensions if and only if its deficiency indices are equal:
\begin{equation}
\label{EqDI}
n_{+}=n_{-}\,.
\end{equation}
\item
Assume that the deficiency indices of the operator \(A\) are equal and non-zero:
\(0<n_{+}=n_{-}\leq\infty\). Choose a pair of non-real conjugated
complex numbers, for example \(\lambda=i\),
\(\overline{\lambda}=-i\). The set of all self-adjoint extensions
of the operator \(A\) is in one-to-one correspondence with the set
of all unitary operators acting from the deficiency subspace
\(\mathcal{N}_i\) into the deficiency subspace
\(\mathcal{N}_{-i}\).
\end{enumerate}
\end{nonumtheorem}
We apply the von Neumann Theorem to the situation where the
operator \(\mathcal{M}_{\textup{min}}\) is taken as the operator
\(A\). Then the equation \[A^\ast{}x=\lambda{}x\] takes the
form
\[\mathcal{M}_{\textup{max}}x=\overline{\lambda}x\,.\] This is the
differential equation
\begin{equation}
\label{DEFEWPr}
-\frac{d\,\,}{dt}\bigg(1-\frac{t^2}{a^2}\bigg)\frac{dx(t)}{dt}=
\overline{\lambda}{}x(t),\quad
 t\in(-a,a),
\end{equation}
under the extra condition \(x(t)\in{}L^2(-a,a).\) In particular,
the dimension of the deficiency space \(\mathcal{N}_{\lambda}\)
coincides with the dimension of the linear space of the set of
solutions of the equation \eqref{DEFEWPr} belongings to
\(L^2(-a,a)\). According to Lemma \ref{InArS}, every solution of
the equation \eqref{DEFEWPr} belongs to \(L^2(-a,a)\). Thus we
prove the following
\begin{lem}
\label{LDefInd}%
 For the operator \(\mathcal{M}_{\textup{min}}\), the deficiency
 indices are:
 \begin{equation}
 \label{dimdo}%
 n_{+}(\mathcal{M}_{\textup{min}})=2,\quad n_{-}(\mathcal{M}_{\textup{min}})=2\,.
\end{equation}
\end{lem}
Thus, the operator \(\mathcal{M}_{\textup{min}}\) is symmetric,
but not self-adjoint, and the set of all its self-adjoint extensions
can by parameterized by the set of all unitary
operators acting from the two-dimensional deficiency subspace \(\mathcal{N}_{i}\)
into the two-dimensional deficiency subspace \(\mathcal{N}_{-i}\). However we
use another parametrization.

\section{Self-adjoint extensions of operators and self-orthogonal subspaces.}

J. von Neumann, \cite{Neu},  reduced the construction of a
self-adjoint extension for a symmetric operator \(A_0\) to an
equivalent problem of construction of an unitary extension of an
appropriate isometric operator - the Caley transform of this
symmetric operator. This approach was also developed by M.\,Stone,
\cite{St}, and then used by many others.

In some situations, it is much more convenient to use the
construction of extensions based on the so called \emph{boundary
forms}. The usage of such construction is especially convenient
for differential operators. The first version of the extension
theory based on abstract symmetric boundary conditions was
developed by J.W.\,Calkin, \cite{Cal}. Afterwards, various
versions of the extension theory of symmetric operators  were developed
in terms
of abstract boundary conditions. The problem
of the descriptions of extensions of symmetric relations
was also considered. See \cite{RoB}, \cite{Koch}, \cite{Br}.

Let \(A\) be a symmetric operator  acting
in a Hilbert space \(\mathfrak{H}\). We assume that the domain of definition
\(\mathcal{D}_A\) of the operator \(A\) is dense in \(\mathfrak{H}\) and that
the operator \(A\) is closed. Since \(A\) is symmetric and densely defined, the adjoint operator \(A^\ast\) exists, and \(A\subseteq A^{\ast}\), that is
\(\mathcal{D}_A\subseteq\mathcal{D}_{A^\ast},\ Ax=A^{\ast}x,\ \forall
x\in\mathcal{D}_A.\) Since \(A\) is closed, the equality \((A^{\ast})^\ast=A\) holds.

We relate
the form \(\Omega\) to the operator \(A\):
\begin{subequations}
\label{HerFor}
\begin{equation}%
\label{HerFor1}
\Omega(x,y)=
\frac{\langle{}A^{\ast}{}x,y\rangle-\langle{}x,A^\ast{}y\rangle}{i}\,,
\quad \Omega:\,\mathcal{D}_{\!A^{{}^{\ast}}}\times{}%
 \mathcal{D}_{\!A^{{}^\ast}}\to\mathbb{C}\,.
\end{equation}%
The form \(\Omega\) is hermitian:
\begin{equation}%
\label{Herm}
\Omega(x,y)=\overline{\Omega(y,x)},\qquad \forall{}x,y\in \mathcal{D}_{\!A^{{}^{\ast}}},
\end{equation}%
and possesses the property
\begin{equation}%
\label{Nuls}
\Omega(x,y)=0,\quad \forall\,x\in{}\mathcal{D}_{A^{{\ast}}},\,y\in{}\mathcal{D}_{A}\,.
\end{equation}%
\end{subequations}
This property allows to consider the form \(\Omega\) as a form on the factor-space \(\mathcal{E}\):
\begin{equation}%
\label{FaSp}%
\mathcal{E}=
\mathcal{D}_{A^\ast}\big%
/\mathcal{D}_{A}\,.
\end{equation}%
We use the same notation for the form induced on the factor space \(\mathcal{E}\):
\begin{equation}%
\label{HerFor2}
\Omega(x,y)=\frac{\langle{}A^{\ast}{}x,y\rangle-\langle{}x,A^\ast{}y\rangle}{i}\,\ccomma
\quad
\Omega:\,\mathcal{E}\times{}\mathcal{E}\to\mathbb{C}\,.
\end{equation}%

\begin{defn}
\label{BoundForm}
The form \(\Omega\), \eqref{HerFor}, is said to be the \emph{boundary form.}
The factor space \(\mathcal{E}\) is said to be the \emph{boundary space}.
\end{defn}According to von Neumann Theorem,
\begin{equation}%
\label{DimBS}
 \dim{}\mathcal{E}=n_{+}+n_{-}\,,
\end{equation}%
where \(n_{+}\) and \(n_{-}\) are deficiency indices of the operator \(A\).

\begin{lem}
\label{LnDe}
The form \(\Omega\) is not degenerate on \(\mathcal{E}\). In other words,
for each non-zero \(x\in\mathcal{E}\), there exists \(y\in\mathcal{E}\)
 such that \(\Omega(x,y)\not=0\,.\)
\end{lem}
\begin{proof} Let \(x\in\mathcal{D}_{A^\ast}\) be given. We assume that
\(\Omega(x,y)=0,\ \forall\,y\in\mathcal{D}_{A^\ast}\). This means that
\(\langle{}x,A^\ast{}y\rangle=\langle{}A^{\ast}{}x,y\rangle,\,
\forall\,y\in\mathcal{D}_{A^\ast}\). The last equality means that \(x\in\mathcal{D}_{(A^\ast)^\ast}\) and \(A^{\ast}x=(A^\ast)^{\ast}x\).
Since \((A^\ast)^{\ast}=A\), we conclude that \(x\in\mathcal{D}_{A}\).
\end{proof}
The definitions of the boundary form and the boundary space can be found in \cite[\S 1]{Str}.\\[2.0ex]

Let \(\mathcal{S}\) be a subspace of the factor space \(\mathcal{E}\):
\begin{subequations}
\label{extsp}
\begin{equation}%
\label{extsp1}
\mathcal{S}\subseteq\mathcal{E}\,.
\end{equation}%
 We identify
\(\mathcal{S}\) with its preimage with respect to the factor-mapping
\(\mathcal{D}_{A^\ast}\to\mathcal{D}_{A^\ast}\big/\mathcal{D}_{A}\,(=\mathcal{E})\)
and use  the same notation \(\mathcal{S}\) for a subspace in \(\mathcal{E}\)
and for its preimage in \(\mathcal{D}_{A}\):
\begin{equation}%
\label{extsp2}
\mathcal{D}_{A}\subseteq\mathcal{S}\subseteq\mathcal{D}_{A^\ast}\,.
\end{equation}%
\end{subequations}
To every \(\mathcal{S}\) satisfying \eqref{extsp2}, an extension of the
operator \(A\) is related. We denote this extension by \(A_{\mathcal{S}}\):
\[\mathcal{D}_{A_{\mathcal{S}}}=\mathcal{S}\,,\quad
A_{\mathcal{S}}\,x={}A^\ast \,x,\ \ \ \forall\,x\in\mathcal{S}. \]
The operator \((A_{\mathcal{S}})^{\ast}\), which is the operator adjoint to
the the operator \(A_{\mathcal{S}}\), is related to the subspace
\(\mathcal{S}^{\bot_\Omega}\):
\begin{equation}%
\label{sadop}
(A_{\mathcal{S}})^\ast=A_{\mathcal{S}^{\bot_\Omega}}\,,
\end{equation}%
where \(\mathcal{S}^{\bot_{\Omega}}\) is the \emph{orthogonal
complement} of the subspace \(\mathcal{S}\) with respect to the
hermitiam form \(\Omega\):
\begin{equation}%
\label{orcom}
\mathcal{S}^{\bot_\Omega}=\lbrace{}x\in\mathcal{E}:\ \ \Omega(x,y)=0\ \ %
 \forall\,y\in\mathcal{S}\rbrace\,.
\end{equation}%
In particular the following result holds:
\begin{lem}
\label{CrSel}
The extension \(A_{\mathcal{S}}\) of the symmetric operator \(A\)
is a self-adjoint operator: \(A_{\mathcal{S}}=(A_{\mathcal{S}})^\ast\),
if and only if the subspace \(\mathcal{S}\) which appears in \eqref{extsp2}
possesses the property:
\begin{equation}%
\label{SeOrSu}
\mathcal{S}=\mathcal{S}^{\bot_\Omega}\,.
\end{equation}%
\end{lem}
\begin{defn}
\label{DeSoSu} The subspace \(\mathcal{S}\) of the boundary space
\(\mathcal{E}\) is said to be
\(\Omega\)-\emph{self-orthogonal} if it possess the
property \eqref{SeOrSu}.
\end{defn}
Thus, \emph{the problem of description of all self-adjoint
extension of a symmetric operator \(A_0\) can be reformulate as
the problem of description of subspaces of the space
\(\mathcal{E}\), \eqref{FaSp}, which are self-orthogonal
with respect to
the (non-degenerated) form \(\Omega\), \eqref{HerFor2}.}

It turns out that  self-orthogonal subspaces exist if and
only if the form \(\Omega\), \eqref{HerFor2}, has equal numbers of
positive and negative squares. (Which conditions is equivalent to
the condition \(n_{+}=n_{-}\).)

\noindent

\section{Self-adjoint extensions of symmetric differential operators.}

The description of self-adjoint extensions of a symmetric operator \(A\) becomes
especially transparent in the case when this symmetric operator is a
formally self-adjoint ordinary differential operator, regular or singular. In this case the
boundary form \(\Omega\), \eqref{HerFor1}, can be expressed in term of
the endpoint forms \([x,y]_{-a}\) and \([x,y]_{a}\), which were introduced in
section \ref{SeBLF}. See Definition \ref{DBBF}. This justifies the terminology introduced in
Definition~\ref{BoundForm}.

We illustrate the situation as applied to the case where the
symmetric operator \(A\) is the minimal differential operator
\(\mathcal{M}_{\textup{min}}\) generated by the formal Legendre differential operator \(M\). Then the adjoint operator
\(A^\ast\) is the maximal differential operator
\(\mathcal{M}_{\textup{max}}\) (See Definitions \ref{MinDO} and
\ref{MaxDO}.)

 The problem of description of self-adjoint
differential operators generated by a given formal differential
operator has the long history. See, for example, \cite{Kr},
\cite[Chapter 5]{Nai}. The book of \cite{DuSch} is the storage of
wisdom in various aspects of the operator theory, in particular is
self-adjoint ordinary differential operators. See especially
Chapter XIII of \cite{DuSch}.

We could incorporate this issue to one or another existing abstract scheme.
 However to
adopt our question to such a scheme one need to agree the
notation, the terminology, etc. This auxiliary work may obscure
the presentation. To make the presentation more transparent, we
prefer to act independently on the existing general considerations.

Let us consider the boundary form \(\Omega_M\), constructed from the operator
\(A=\mathcal{M}_{\textup{min}}\) according to \eqref{HerFor1}. Using Theorem \ref{MSt} we conclude that
\begin{equation}%
\label{LRBFo1}
\Omega_M(x,y)=\frac{\mathcal{\langle{}M}_{\textup{max}}\,x,y\rangle-
\langle{}x{},\mathcal{M}_{\textup{max}}\,y\rangle}{i}\ccomma\quad \forall\,x,y\in\mathcal{D}_{\mathcal{M}_{\textup{max}}}.
\end{equation}%
The appropriate boundary space \(\mathcal{E}_M\) is:
\begin{equation}
\label{BSLO}
\mathcal{E}_M=\mathcal{D}_{\mathcal{M}_{\textup{max}}}\big/
\mathcal{D}_{\mathcal{M}_{\textup{min}}}.
\end{equation}
According to \eqref{DimBS} and Lemma \ref{LDefInd},
\begin{equation}%
\label{DimFS}
\dim\,\mathcal{E}_M=4\,.
\end{equation}%
By Theorem \ref{GreFT}, the boundary form \(\Omega_M\)
can be expressed in the term of the end point forms
\([x,y]_{-a},\,[x,y]^{a}\):
\begin{equation}%
\label{LRBFo2}%
\Omega_M(x,y)=\frac{[x,y]^{a}-[x,y]_{-a}}{i}\,\ccomma\quad \forall\,x,y\in\mathcal{D}_{\mathcal{M}_{\textup{max}}}.
\end{equation}
To make calculation explicit, we choose a special basis in the
space \(\mathcal{E}_M\).  The asymptotic behavior of solutions
of the equation \(Lx=0\) near the endpoints of the interval
\((-a,a)\), described in Lemma \ref{ABSNS}, prompts us the choice
of such a basis.

Let us choose and fix smooth real valued functions \(\varphi_{-}(t), \psi_{-}(t),
\varphi_{+}(t)\), \(\psi_{+}(t)\) defined on the interval \((-a,a)\) such that
\begin{subequations}
\label{SpBa}
\begin{alignat}{4}
&\varphi_{-}(t)=1,\ &-a<t<-a/2,&\quad &\varphi_{-}(t)&=0, \
&a/2<t<a\,,\\
&\psi_{-}(t)=\ln(a+t),\ &-a<t<-a/2,&\quad &\psi_{-}(t)&=0, \
&a/2<t<a\,,\\
 &\varphi_{+}(t)=0, \ &-a<t<-a/2\,,&\quad &\varphi_{+}(t)&=1,\
&a/2<t<a\,,\\
&\psi_{+}(t)=0, \  &-a<t<-a/2\,,&\quad &\psi_{+}(t)&=\ln(a-t),\
&a/2<t<a\,.
\end{alignat}
\end{subequations}
It is clear that
\begin{equation}
\label{Bel}
\varphi_{-}\in\mathcal{D}_{\mathcal{M}_{\textup{max}}},\ \  \psi_{-}\in\mathcal{D}_{\mathcal{M}_{\textup{max}}}, \ \
\varphi_{+}\in\mathcal{D}_{\mathcal{M}_{\textup{max}}}, \ \ \psi_{+}\in\mathcal{D}_{\mathcal{M}_{\textup{max}}}.
\end{equation}

The next calculations are based on the representation \eqref{LRBFo2}.
Since the end point forms \([x,y]_{-a}, [x,y]^{a}\) are skew-hermitian,
 then \(\Omega_M(\chi,\chi)=0\) for each real valued function
 \(\chi\in\mathcal{D}_{\mathcal{M}_{\textup{max}}}\). In
particular,
\begin{subequations}
\label{CBF}
\begin{equation}
\label{CBF1}
\Omega_M(\chi,\chi)=0,\ \ \textup{if \(\chi\) is one of the functions } \varphi_{-},\psi_{-},
\varphi_{+},\psi_{+}\,.
\end{equation}
It is clear that
\begin{equation}
\label{CBF2}
\Omega_M(\chi_{-},\chi_{+})=0,\ \ \textup{if \(\chi_{\pm}\) is one of the functions }
\varphi_{\pm},\psi_{\pm}\,.
\end{equation}
 Direct calculation shows that
\begin{equation}
\label{CBF3}
\Omega_M(\varphi_{-},\psi_{-})=\frac{2i}{a},\quad \Omega_M(\varphi_{+},\psi_{+})=-\frac{2i}{a}\,.
\end{equation}
\end{subequations}
Thus, the Gram matrix (with respect to the hermitian form \(\Omega_M\))
 of the vectors \(\varphi_{-}\), \(\psi_{-}\),
 \(\varphi_{+}\), \(\psi_{+}\) is:
 \begin{equation}
 \label{GrMa}
 \frac{a}{2}  \cdot%
 \begin{bmatrix}
\Omega_M(\varphi_{-},\varphi_{-})&\Omega_M(\varphi_{-},\psi_{-})%
&\Omega_M(\varphi_{-},\varphi_{+})&\Omega_M(\varphi_{-},\psi_{+})\\
\Omega_M(\psi_{-},\varphi_{-})&\Omega_M(\psi_{-},\psi_{-})%
&\Omega_M(\psi_{-},\varphi_{+})&\Omega_M(\psi_{-},\psi_{+})\\
\Omega_M(\varphi_{+},\varphi_{-})&\Omega_M(\varphi_{+},\psi_{-})%
&\Omega_M(\varphi_{+},\varphi_{+})&\Omega_M(\varphi_{+},\psi_{+})\\
\Omega_M(\psi_{+},\varphi_{-})&\Omega_M(\psi_{+},\psi_{-})%
&\Omega_M(\psi_{+},\varphi_{+})&\Omega_M(\psi_{+},\psi_{+})
 \end{bmatrix}=J,
 \end{equation}
 where
 \begin{equation}
 \label{IndMe}
 J=
\begin{bmatrix}
0 &i & 0&0 \\
 -i & 0& 0&0 \\
 0 & 0& 0& i\\
 0& 0& -i&0
\end{bmatrix}\,.
\end{equation}
The rank of the Gram matrix is is equal to the dimension of the space~\(\mathcal{E}_M\):
\begin{equation}
\label{eqr}
\textup{rank}\,J=\dim{}\mathcal{E}_M\,(=4)\,.
\end{equation}
\begin{lem}
\label{GenBS}
The functions \(\varphi_{-}\), \(\psi_{-}\),
 \(\varphi_{+}\), \(\psi_{+}\) generate the boundary space
 \(\mathcal{E}_M\).
 \end{lem}
 \begin{proof}
 Lemma \eqref{GenBS} is a consequence of \eqref{Bel} and of         the equality \eqref{eqr}.
 \end{proof}
 \begin{lem}
 \label{DDMo}
The domain of definition \(\mathcal{D}_{\mathcal{M}_{\textup{min}}}\) of the minimal
differential operator \(\mathcal{M}_{\textup{min}}\) can be characterized by means
of the conditions:
\begin{multline}
\label{DoDeMiO}
\mathcal{D}_{\mathcal{M}_{\textup{min}}}=
\big\lbrace{}x(t)\in\mathcal{D}_{\mathcal{M}_{\textup{max}}} :
\\
\Omega_M(x,\varphi_{-})=0,\ \ \Omega_M(x,\psi_{-})=0, \ \
\Omega_M(x,\varphi_{+})=0,\ \ \Omega_M(x,\psi_{+})=0
\big\rbrace.
\end{multline}
\end{lem}
\begin{proof} According to Lemma \ref{GenBS}, from \eqref{DoDeMiO} it follows that
\(\Omega_M(x,y)=0,\,\forall\,y\in\mathcal{M}_{\textup{max}}\). Now we refer to
Lemma \ref{LnDe} and to Theorem \ref{MSt} taking the operator \(\mathcal{M}_{\textup{min}}\) as the operator \(A\).
\end{proof}
\begin{lem}
\label{OJOrt}%
Let \(\Omega_M\) be a bilinear form in the boundary space \(\mathcal{E}\) defined by
\eqref{LRBFo1}, and \(J\) be the matrix \eqref{IndMe}.

The vector \(x^{1}=\alpha_{-}^{1}\varphi_{-}+\beta_{-}^{1}\psi_{-}+
\alpha_{+}^{1}\varphi_{+}+\beta_{+}^{1}\psi_{+}\in\mathcal{E}_L\) is \(\Omega_M\)~-orthogonal
to the vector \(x^{2}=\alpha_{-}^{2}\varphi_{-}+\beta_{-}^{2}\psi_{-}+
\alpha_{+}^{2}\varphi_{+}+\beta_{+}^{2}\psi_{+}\in\mathcal{E}_L\), that is
\begin{subequations}
\label{Ort}
\begin{equation}%
\label{Ort1}
\Omega_M(x^1,x^2)=0,
\end{equation}
 if and only
if the vector-row \(v_{x^1}=[\alpha_{-}^{1},\beta_{-}^{1},\alpha_{+}^{1},\beta_{+}^{1}]\in\mathcal{V}\)
is \(J\)-orthogonal to the vector-row \(v_{x^2}=[\alpha_{-}^{2},\beta_{-}^{2},\alpha_{+}^{2},\beta_{+}^{2}]\in\mathcal{V}\), that is
\begin{equation}%
\label{Ort2}
v_{x^1}J\,v_{x^2}^{\,\ast}=0\,,
\end{equation}
\end{subequations}
where \(\mathcal{V}\) is the space \(\mathbb{C}^4\) of vector-rows equipped by the standard hermitian metric, and the star \(\ast\) is the Hermitian conjugation.
\end{lem}
Thus, the problem of description of self-adjoint extensions of the operator
\(\mathcal{M}_{\textup{min}}\) is equivalent to the problem of description
of \(\Omega_M\)-self-orthogonal\,%
\footnote{%
As soon as the notion of \(J\)-orthogonality of two vectors is introduced,
\eqref{Ort2}, the notions of \(J\)-orthogonal complement and \(J\)-self-orthogonal
subspaces can be introduced as well.
} %
subspaces in \(\mathcal{E}\), which in its turn is
equivalent to the problem of description of \(J\)-self-orthogonal subspaces in
\(\mathbb{C}^4\). The last problem is a problem of the indefinite linear algebra
and admits an explicit  solutions. We set
\begin{subequations}%
\label{Jpro}
\begin{equation}
\label{Jpro1}
P_{+}=\frac{1}{2}(I+J),\quad P_{-}=\frac{1}{2}(I-J)\,,
\end{equation}
More explicitly,
\begin{equation}
\label{Jpro2}
P_{+}=\frac{1}{2}
\begin{bmatrix}
1 &i & 0&0 \\
 -i & 1& 0&0 \\
 0 & 0& 1& i\\
 0& 0& -i&1
\end{bmatrix}\,,\quad
P_{-}=\frac{1}{2}
\begin{bmatrix}
1 &-i & 0&0 \\
 i & 1& 0&0 \\
 0 & 0& 1& -i\\
 0& 0& i&1
\end{bmatrix}\,.
\end{equation}
\end{subequations}
The matrix \(J\), \eqref{IndMe}, possesses the properties
\begin{equation*}
J=J^{\ast},\quad J^2=I.
\end{equation*}
Therefore the matrices \(P_{+},\ P_{-}\), \eqref{Jpro1},  possess  the properties
\begin{align}
\label{PrPr}
P_{+}^2=P_{+},\quad P_{-}^2&=P_{-}, \quad P_{+}=P_{+}^{\,\ast},\quad
\quad P_{-}=P_{-}^{\,\ast},\\
  P_{+}P_{-}&=0, \quad P_{+}+P_{-}=I\,.
\end{align}
In other words, the matrices \(P_{+},\ P_{-}\) are orthogonal projector matrices.
These matrices project the space \(\mathcal{V}\) onto subspaces
\(\mathcal{V}_{+}\) and \(\mathcal{V}_{-}\):
\begin{equation}
\label{PMSu}%
\mathcal{V}_{+}=\mathcal{V}P_{+},\ \mathcal{V}_{-}=\mathcal{V}P_{-}\,.
\end{equation}
These  subspaces are orthogonally complementary:
\begin{equation}
\label{PMSu1}%
\mathcal{V}_{+}\oplus\mathcal{V}_{-}=\mathcal{V}\,.
\end{equation}
The vector rows
\begin{subequations}
\label{bas}
\begin{alignat}{2}
\label{bas1}
e^1_{+}&=[1,\phantom{-}i,0,0],\quad& e^2_{+}&=[0,0,1,\phantom{-}i]\\
\intertext{and}
\label{bas2}
e^1_{-}&=[1,-i,0,0],\quad & e^2_{-}&=[0,0,1,-i]%
\end{alignat}
\end{subequations}
form orthogonal\,%
\footnote{\,%
 In the standard scalar product on
\(\mathcal{V}=\mathbb{C}^4\).} %
 bases in \(\mathcal{V}_{+}\) and \(\mathcal{V}_{-}\) respectively.

 It turns out that \(J\)-self-orthogonal subspaces of the space \(\mathcal{V}\)
are in one-to-one correspondence with unitary operators acting from
\(\mathcal{V}_{+}\) onto \(\mathcal{V}_{-}\).
\begin{defn}
Let \(U\) be an unitary operator  acting from \(\mathcal{V}_{+}\) onto
\(\mathcal{V}_{-}\). As the vector-row \(v\) runs  over the whole subspace
\(\mathcal{V}_{+}\), the vector \(v+vU\) runs over a subspace
of the space \(\mathcal{V}\). This subspace is denoted by \(\mathcal{S}_U\):
\begin{equation}
\mathcal{S}_U=\big\lbrace{}v+vU\big\rbrace,\quad\textup{where }v%
\textup{ runs over the whole }\mathcal{V}_{+}\,.
\end{equation}
\end{defn}
\newpage
\begin{lem}{\ }\\
\label{DJSO}
\begin{enumerate}
\item
Let \(U\) be an unitary operator  acting from \(\mathcal{V}_{+}\) onto
\(\mathcal{V}_{-}\).
Then the subspace \(\mathcal{S}_U\)  is \(J\)-self-orthogonal, that is
\[\mathcal{S}_U=\mathcal{S}_U^{\bot_J}\,.\]
\item
 Every \(J\)-self-orthogonal subspace \(\mathcal{S}\)
of the space \(\mathcal{V}\) is of the form \(\mathcal{S}_U\):
\[\mathcal{S}=\mathcal{S}_U\]
for some unitary operator \(U:\,\mathcal{V}_{+}\to\mathcal{V}_{-}\).
\item
 The correspondence between \(J\)-self-orthogonal subspaces and
unitary operators acting from \(\mathcal{V}_{+}\) onto
\(\mathcal{V}_{-}\) is one-to-one;
\[(U_1=U_2)\Leftrightarrow{}(\mathcal{S}_{U_1}=\mathcal{S}_{U_2})\,.\]
\end{enumerate}
\end{lem}
\begin{proof}{\ }\\
\textbf{1}.\,The mapping \(v\to{}v+Uv\) is one-to-one mapping from \(\mathcal{V}_{+}\)
onto \(\mathcal{S}_U\). Indeed, this mapping is surjective by definition of
the subspace \(\mathcal{S}_U\). This mapping is also injective.
 The equality \(v+Uv=0\) implies that \(v=Uv=0\) since\,%
 \footnote{%
 Recall that
 \(v\in\mathcal{V}_{+},\,Uv\in\mathcal{V}_{-}\), %
and \(\mathcal{V}_{+}\bot \mathcal{V}_{-}\).
 } %
  \(v\bot\,Uv\). In particular, \(\dim\mathcal{S}_U=\dim\mathcal{V}_{+}\,(=2)\).

If \(v_1\) and \(v_2\) are two arbitrary vectors from \(\mathcal{V}_{+}\),
then the vectors \(w_1=v_1+v_1U\) and \(w_2=v_2+v_2U\) are \(J\)-orthogonal:
\(w_1Jw_2^{\ast}=0\). Indeed, since \(J=P_{+}-P_{-}\) and %
\(v_k=v_kP_{+}, v_kU=v_kUP_{-}\,,k=1,2\), then, using the properties \eqref{PrPr}
of \(P_{+}\) and  \(P_{-}\), we obtain
\begin{align*}%
w_1Jw_2^{\ast}=(v^1P_{+}&+v^1UP_{-})(P_{+}-P_{-})
(P_{+}^\ast{}v_2^\ast+P_{-}^\ast{}U^\ast{}v_2^\ast)=\\
&=v_1v_2^\ast{}-v_1UU^\ast{}v_2^\ast\,.
\end{align*}
Since the unitary operator \(U\) preserves the scalar product, then
\(v_1v_2^{\ast}=v_1UU^\ast{}v_2^{\ast}\), hence \(w_1Jw_2^{\ast}=0\).
Thus, \(\mathcal{S}_U\subseteq(\mathcal{S}_U)^{\bot_J}\).
(The symbol \(\bot_J\) means \(J\)-orthogonal complement.)
Since the Hermitian form \((v_1,v_2)\to{}v_1Jv_2^\ast\) is non-degenerate
on \(\mathcal{V}\), then \(\dim(\mathcal{S}_U^{\bot_J})=
\dim{\mathcal{V}}-\dim{\mathcal{S}_U}\). Because %
\(\dim{\mathcal{V}}-\dim{\mathcal{S}_U}=\dim{\mathcal{S}_U}\), we have
\(\dim{\mathcal{S}_U}=\dim(\mathcal{S}_U^{\bot_J})\). Hence,
\(\mathcal{S}_U=(\mathcal{S}_U)^{\bot_J}\), i.e. the subspace \(\mathcal{S}_U\)
is \(J\)-self-orthogonal.\\
\textbf{2}.\,Let \(\mathcal{S}\) be a \(J\)-self-orthogonal subspace.
If \[v\in\mathcal{S},\,v=v_{1}+v_{2},\,v_{1}\in\mathcal{V}_{+}\,,\,v_{2}\in\mathcal{V}_{-},\]
then the condition \(v{\bot_J}v=0\), that is the condition \(vJv^\ast=0\)
means that \(v_1v_1^\ast=v_2v_2^\ast\). Therefore, if \(v_1=0\), then also
\(v=0\). This means that the projection mapping \(v\to{}vP_{+}\), considered
as a mapping from \(\mathcal{S}\to\mathcal{V}_{+}\), is injective.
For \(J\)-self-orthogonal subspace \(\mathcal{S}\) of the space \(\mathcal{V}\),
the equality \(\dim\mathcal{S}=\dim\mathcal{V}-\dim\mathcal{S}\) holds.
Hence \(\dim\mathcal{S}=\dim\mathcal{V}_{+}\). Therefore, the injective linear
mapping \(v\to{}P_{+}\) is surjective. The inverse mapping
is defined on the whole subspace \(\mathcal{V}_{+}\) and can by presented
in the form \(v=v_1+v_1U\), where \(U\) is a linear operator acting
from \(\mathcal{V}_{+}\) into \(\mathcal{V}_{-}\). This  mapping \(v_1\to{}v_1+v_1U\)
maps the subspace \(\mathcal{V}_{+}\) onto the subspace \(\mathcal{S}\).

 Since \(vJv^\ast=0\),
then \(v_1v_1^\ast=v_2v_2^\ast\), where \(v_2=v_1U\). Since \(v_1\in\mathcal{V}_{+}\)
is arbitrary, this means that the operator \(U\) is isometric. Since
\(\dim{}\mathcal{V}_{+}=\dim{}\mathcal{V}_{-}\), the operator \(U\) is unitary.
Thus, the originally given \(J\)-self-orthogonal subspace \(\mathcal{S}\)
is of the form \(\mathcal{S}_U\), where \(U\) is an unitary operator acting
from \(\mathcal{V}_{+}\) to \(\mathcal{V}_{-}\).\\
\textbf{3}.
The coincidence \(\mathcal{S}_{U_1}=\mathcal{S}_{U_2}\) means that
every vector of the form \(v_1+v_1U_1\), where \(v_1\in\mathcal{V}_{+}\)
can also be presented in the form \(v_2+v_2U_2\) with some \(v_2\in\mathcal{V}_{+}\):
\[v_1+v_1U_1=v_2+v_2U_2\,.\]
Since \(v_1,\,v_2\in\mathcal{V}_{+},\ v_1U_1,\,v_1U_2\in\mathcal{V}_{-}\),
then \(v_1=v_2\), and \(v_1U_1=v_1U_2\). The equality \(v_1U_1=v_1U_2\)
for every \(v_1\in\mathcal{V}_{+}\) means that \(U_1=U_2\).
Thus, \((\mathcal{S}_{U_1}=\mathcal{S}_{U_2})\Rightarrow(U_1=U_2)\).
\end{proof}
Choosing the orthogonal bases \eqref{bas} in the subspaces \(\mathcal{V}_{+}\)
and \(\mathcal{V}_{+}\), we represent an unitary operator \(U\) by the appropriate
unitary matrix:
\begin{alignat*}{2}
e^1_{+}U&=&\ e^1_{-}u_{11}&+e^2_{-}u_{21},\\
e^2_{+}U&=&\ e^1_{-}u_{12}&+e^2_{-}u_{22}.
\end{alignat*}
The
following result is a  reformulation of Lemma \ref{DJSO}:
\begin{lem}
\label{BSOS}%
Let \(\mathcal{V}\) be the space \(\mathbb{C}^4\) of \(four\) vector-rows,
\(J\) be a matrix of the form \eqref{IndMe}. With every \(2\times2\) matrix %
\( U=\|u_{pq}\|_{1\leq{}p,q\leq{}2}\), we
associate the pair of vectors \(v^1(U),\,v^2(U)\):
\begin{subequations}
\label{BSOSu}
\begin{alignat}{2}
\label{BSOSu1}
v^1(U)&=e^1_{+}+&\ e^1_{-}u_{11}&+e^2_{-}u_{21},\\
\label{BSOSu2}
v^2(U)&=e^2_{+}+&\ e^1_{-}u_{12}&+e^2_{-}u_{22},
\end{alignat}
\end{subequations}
where \(e^{\,k}_{\pm},\,k=1,2,\)
are the vector-rows of the form \eqref{bas},
and the subspace \(\mathcal{S}_U\) of \(\mathcal{V}\) is the linear
hull of the vectors \(v^1(U),\,v^2(U)\),
\begin{equation*}
\mathcal{S}_U=\textup{hull}(v^1(U),\,v^2(U))\,.
\end{equation*}
\begin{enumerate}
\item
If the matrix \(U\) is unitary, then the vectors
\(v^1(U),\,v^2(U)\)
are linearly independent, and the subspace  \(\mathcal{S}_U\)
is \(J\)-self-orthogonal.
\item
Let \(\mathcal{S}\) be a \(J\)-self-orthogonal subspace
of the space \(\mathcal{V}\). Then \(\mathcal{S}=\mathcal{S}_U\)
for some  an unitary matrix \(U\).
\item
For  unitary matrices \(U_1,\,U_2\),
\begin{equation*}
(\mathcal{S}_{U_1}=\mathcal{S}_{U_2})\Leftrightarrow(U_1=U_2)\,.
\end{equation*}
\end{enumerate}
\end{lem}
The "coordinate" form of the  vectors \(v^1(U),\,v^2(U)\) is:
\begin{alignat}{5}
\label{Coorv}
v^1(U)=%
&\big[&1+u_{11}&,&\,\,\,i(1-u_{11})&,&\,u_{21}\,\,\,\,&,&-iu_{21}\,\,\,\, &\big],
\notag\\
\raisetag{100pt}
v^2(U)=%
&\big[&u_{12}\,\,\,\,&,&\,-iu_{12}\,\,\,\,&,&\,\,1+u_{22}&,&\,\,i(1-u_{22})&\big].
\raisetag{100pt}
\end{alignat}

Taking in account Lemma \ref{OJOrt}, we formulate the following result
\begin{lem}%
\label{ESoS}%
Let us associate
 the pair of vectors \(d^1(U),\,d^2(U)\in\mathcal{E}_M\)
 with every \(2\times2\) matrix \( U=\|u_{pq}\|_{1\leq{}p,q\leq{}2}\):
 \begin{subequations}%
 \label{VGeE}%
 \begin{align}%
 \label{VGeE1}%
 d^1(U)=&\,(1+u_{11})\varphi_{-}+i(1-u_{11})\psi_{-}+u_{21}\varphi_{+}-iu_{21}\psi_{+}\,,\\
 \label{VGeE2}%
 d^2(U)=&\,u_{12}\varphi_{-}-iu_{12}\psi_{-}+(1+u_{22})\varphi_{+}+i(1-u_{22}\psi_{+}\,,
 \end{align}%
 \end{subequations}%
 where the functions \(\varphi_{\pm},\,\psi_{\pm}\) are defined in \eqref{SpBa}.
The subspace \(\mathcal{G}_U\) of the space \(\mathcal{E}_M\) is defined as
the linear hull of the vectors \(d^1(U),\,d^2(U)\):
\begin{equation}
\label{SSU}
\mathcal{G}_U=\textup{hull}\,(d^1(U),\,d^2(U))\,.
\end{equation}
\begin{enumerate}
\item
If the matrix \(U\) is unitary, then  the subspace  \(\mathcal{S}=\mathcal{G}_U\)
is \(\Omega_M\)-self-orthogonal.
\item
Let \(\mathcal{S}\) be a \(\Omega_M\)-self-orthogonal subspace
of the space \(\mathcal{E}_M\). Then \(\mathcal{S}=\mathcal{G}_U\)
for some  an unitary matrix \(U\).
\item
For  unitary matrices \(U_1,\,U_2\),
\begin{equation*}
(\mathcal{G}_{U_1}=\mathcal{G}_{U_2})\Leftrightarrow(U_1=U_2)\,.
\end{equation*}
\end{enumerate}
\end{lem}%
It is clear that a subspace \(\mathcal{S}\subseteq\mathcal{E}_M\) is
an \(\Omega_M\)-self-orthogonal subspace if and only if its \(\Omega_M\)-orthogonal
complement \(\mathcal{S}^{\bot_{\Omega_M}}\) is an \(\Omega_M\)-self-orthogonal subspace.
The subspace \((\mathcal{S}_U)^{\bot_{\Omega_M}}\) can be described as:
\[(\mathcal{S}_U)^{\bot_{\Omega_M}}=\big\lbrace{}x\in\mathcal{E}_M:
\Omega_{M}(x,d^1(U))=0,\,\Omega_{M}(x,d^2(U))=0\big\rbrace\,,\]
where \(d^1,\,d^2\) are defined in \eqref{VGeE}, \eqref{SpBa}.
Thus Lemma \ref{ESoS} can be reformulated in the following way:%
\begin{lem}
\label{ODSS}
Let us associate
the pair of vectors \(d^1(U),\,d^2(U)\)
with every \(2\times2\) matrix \( U=\|u_{pq}\|_{1\leq{}p,q\leq{}2}\)
by \eqref{VGeE}, \eqref{SpBa}.
The subspace  \(\mathcal{O}_U\) is defined as
\begin{equation}%
\label{OODop}%
\mathcal{O}_U=\big\lbrace{}x\in\mathcal{E}_L:
\,\Omega_{M}(x,d^1(U))=0,\,\Omega_{M}(x,d^2(U))=0\big\rbrace\,.
\end{equation}%
\begin{enumerate}
\item
If the matrix \(U\) is unitary, then  the subspace  \(S=\mathcal{O}_U\)
is \(\Omega_M\)-self-orthogonal.
\item
Let \(\mathcal{S}\) be a \(\Omega_M\)-self-orthogonal subspace
of the space \(\mathcal{E}_M\). Then \(\mathcal{S}=\mathcal{O}_U\)
for some  an unitary matrix \(U\).
\item
For  unitary matrices \(U_1,\,U_2\),
\begin{equation*}
(\mathcal{O}_{U_1}=\mathcal{O}_{U_2})\Leftrightarrow(U_1=U_2)\,.
\end{equation*}
\end{enumerate}
 \end{lem}

 Thus there is one-to-one correspondence between the set of all \(2\times2\) unitary
 matrices \(U=\|u_{pq}\|_{1\leq{}p,q\leq{}2}\) and the set of all \(\Omega_M\)-self-orthogonal
 subspaces \(\mathcal{S}\) of the space \(\mathcal{E}_M=
\mathcal{D}_{\mathcal{M}_{\textup{max}}}\big/%
\mathcal{D}_{\mathcal{M}_{\textup{min}}}\). This correspondence is described as
\begin{equation}%
\label{Corre}%
\mathcal{S}=\mathcal{O}_U,
\end{equation}%
where \(\mathcal{O}_U\) is defined in \eqref{OODop}, \eqref{VGeE}, \eqref{SpBa}.

On the other hand, the subspaces of the space  \(\mathcal{E}_M=
\mathcal{D}_{\mathcal{M}_{\textup{max}}}\big/%
\mathcal{D}_{\mathcal{M}_{\textup{min}}}\)
which are self-orthogonal with respect to the Hermitian form \(\Omega_M\), \eqref{LRBFo1},
are in one-to-one correspondence to self-adjoint differential  operators
generated by the formal differential operator \(M\), \eqref{DiffExp}.
Every self-adjoint differential operators \(\mathcal{M}\)
generated by the formal differential operator \(M\) is the \emph{restriction} of the
maximal differential operator   \(\mathcal{M}_{\textup{max}}\), \eqref{MaxDO},
on the appropriate domain of definition. According to Lemma~\ref{CrSel},
as applied to the operators \(A=\mathcal{M}_{\textup{min}},\,A^{\ast}=\mathcal{M}_{\textup{max}}\),
the domains of definition of a self-adjoint extension \(\mathcal{S}\) of
the operator \(\mathcal{M}_{\textup{min}}\) are those subspaces \(\mathcal{S}\):
\begin{equation}
\mathcal{D}_{\mathcal{M}_{\textup{min}}}\subseteq\mathcal{S}%
\subseteq\mathcal{D}_{\mathcal{M}_{\textup{max}}}
\end{equation}
which are self-orthogonal with respect to the Hermitian form \(\Omega_M\),
\eqref{LRBFo}. According to Lemma \ref{ODSS}, \(\Omega_M\)-self-orthogonal
subspaces \(\mathcal{S}\) can be described by means of the conditions
\begin{equation}
\mathcal{S}=\big\lbrace{}x(t)\in\mathcal{D}_{\mathcal{M}_{\textup{max}}}:\,\,%
\Omega_M(x,d^1(U))=0,\,\,\Omega_M(x,d^2(U))=0\big\rbrace\,,
\end{equation}
where \(d^1(U),\,d^2(U)\) are the same that in \eqref{VGeE},\,\eqref{SpBa},
\(U\) is an unitary \(2\times2\) matrix.

\section{Description of the selfadoint extensions \(\boldsymbol{\mathcal{M}_U}\) in terms of the end point linear forms.}
The conditions \(\Omega_M(x,d^1(U))=0,\,\,\Omega_M(x,d^2(U))=0\)
may be interpreted as a \emph{boundary conditions} posed on functions
\(x\in\mathcal{D}_{\mathcal{M}_{\textup{max}}}\). Let us present these conditions
in more traditional form.
\begin{defn}
\label{DBLF}
For each \emph{fixed} \(y\in\mathcal{D}_{\mathcal{M}_{\textup{max}}}\), the expressions
\([x,y]_{-a}\) and \([x,y]^{a}\), considered as function of \(x\), are linear forms
defined on \(\mathcal{D}_{\mathcal{M}_{\textup{max}}}\). These forms are said to be
\emph{the end point linear forms related to the differential operator~\(M\).}
\end{defn}

In view of \eqref{Bel}, all four endpoint linear forms
\begin{equation}
\label{BoLF}
[x,\varphi_{-}]_{-a},\quad [x,\psi_{-}]_{-a},\quad [x,\varphi_{+}]^{a},\quad [x,\psi_{+}]^{a}
\end{equation}
are well defined for \(x\in\mathcal{D}_{\mathcal{M}_{\textup{max}}}\).
\begin{lem}{\ }\\[-2.0ex]
\label{EEBF}
\begin{enumerate}
\item For every \(x\in\mathcal{D}_{\mathcal{M}_{\textup{max}}}\),
the end point linear forms \([x,\varphi_{-}]_{-a},\, [x,\psi_{-}]_{-a}\),\, \([x,\varphi_{+}]^{a},\, [x,\psi_{+}]^{a}\)
can be expressed as:
\begin{subequations}
\label{ECBF}
\begin{align}
[x,\varphi_{-}]_{-a}=&-\frac{2}{a}b_{-a}(x),\\
[x,\varphi_{+}]^{a}=&\phantom{-\,}\frac{2}{a}b_{a}(x),\\
[x,\psi_{-}]_{-a}=&-\frac{2}{a}c_{-a}(x),\\
[x,\psi_{+}]^{a}=&\phantom{-\,}\frac{2}{a}c_{a}(x),
\end{align}
\end{subequations}
where
\begin{subequations}
\label{GBCo}
\begin{align}%
b_{-a}(x)&=\lim_{t\to{}{-a+0}}(t+a)\frac{dx(t)}{dt},\label{GBCo1}\\
b_{a}(x)&=\lim_{t\to{}{a-0}}(t-a)\frac{dx(t)}{dt},\label{GBCo2}\\
c_{-a}(x)&=\lim_{t\to{}{-a+0}}\bigg((t+a)\ln(a+t)\frac{dx(t)}{dt}-x(t)\bigg),%
\label{GBCo3}\\
c_{a}(x)&=\lim_{t\to{}{a-0}}\bigg((t-a)\ln(a-t)\frac{dx(t)}{dt}-x(t)\bigg)\,.\label{GBCo4}
\end{align}
In particular,  the limits exist in \eqref{GBCo}.
\end{subequations}
\item
The end poins linear forms \([x,\varphi_{-}]^{a},\, [x,\psi_{-}]^{a}\),\, \([x,\varphi_{+}]_{-a},\, [x,\psi_{+}]_{-a}\) vanish identically on \(\mathcal{D}_{\mathcal{M}_{\textup{max}}}\).
\end{enumerate}
\end{lem}
\begin{proof} Let us introduce
\begin{subequations}
\label{GBCTr}
\begin{alignat}{2}%
\label{GBCTr1}
b_{-a}(x)&=\frac{ia}{2}\,\Omega_M(x,\varphi_{-}),&\quad
c_{-a}(x)&=\frac{ia}{2}\,\Omega_M(x,\psi_{-}),\\[1.0ex]
\label{GBCTr2}
b^{a}(x)&=\frac{ia}{2}\,\Omega_M(x,\varphi_{+}),&\quad
c^{a}(x)&=\frac{ia}{2}\,\Omega_M(x,\psi_{+}),
\end{alignat}%
\end{subequations}
From \eqref{LRBFo2} it follows that the equalities \eqref{ECBF} hold.
The existence of the limits in \eqref{GBCo} follows from
Lemma \ref{EBV} applied to the functions \(x(t)\) and \(y(t)=\varphi_{\pm}(t)\)
or \(y(t)=\psi_{\pm}(t)\). The equalities \eqref{GBCTr} can be obtained by
the direct computation using the explicit expressions \eqref{SpBa} for the
functions \(\varphi_{\pm}(t),\,\psi_{\pm}(t)\).
\end{proof}
\begin{rem}%
\label{GeBV}
 The values \(b_{-a}(x),\,c_{-a}(x)\),\,
\(b_{a}(x),\,c_{a}(x)\) may be considered as generalized boundary
values related to the function
\(x(t)\in\mathcal{D}_{\mathcal{M}_{\textup{max}}}\) at the
end points \(-a\) and \(a\) of the interval \((-a,a)\).
\end{rem}
In view of \eqref{DoDeMiO} and \eqref{GBCTr}, Lemma \ref{DDMo} can be reformulated as follow.
\begin{thm}
\label{DDMor}
The domain of definition \(\mathcal{D}_{\mathcal{M}_{\textup{min}}}\) of the minimal
differential operator \(\mathcal{M}_{\textup{min}}\) can be characterized by means
of the boundary conditions:
\begin{multline}
\label{DoDeMiOr}
\mathcal{D}_{\mathcal{M}_{\textup{min}}}=
\big\lbrace{}x(t)\in\mathcal{D}_{\mathcal{M}_{\textup{max}}} :
\\
b_{-a}(x)=0,\ \ b^{a}(x)=0, \ \
c_{-a}(x)=0,\ \ c^{a}(x)=0
\big\rbrace.
\end{multline}
\end{thm}
{\ }\\
 Due to \eqref{GBCTr}, the equality \eqref{GrMa} can be rewritten
as
\begin{equation}
 \label{Rewr}
 \begin{bmatrix}
b_{-a}(\varphi_{-})&c_{-a}(\varphi_{-})&b_{-a}(\varphi_{-})&c_{a}(\varphi_{-})\\
b_{-a}(\psi_{-})&c_{-a}(\psi_{-})&b_{-a}(\psi_{-})&c_{a}(\psi_{-})\\
b_{-a}(\varphi_{+})&c_{-a}(\varphi_{+})&b_{-a}(\varphi_{+})&c_{a}(\varphi_{+})\\
b_{-a}(\psi_{+})&c_{-a}(\psi_{+})&b_{-a}(\psi_{+})&c_{a}(\psi_{+})
 \end{bmatrix}=
\begin{bmatrix}
0&-1&\phantom{-}0&\phantom{-}0\\
1&\phantom{-}0&\phantom{-}0&\phantom{-}0\\
0&\phantom{-}0&\phantom{-}0&-1\\
0&\phantom{-}0&\phantom{-}1&\phantom{-}0
 \end{bmatrix}\,.
 \vspace*{0.8ex}
 \end{equation}

According to \eqref{GBCTr}, the equalities
\(\Omega_M(x,d^1(U))=0,\,\Omega_M(x,d^2(U))=0\) take the form
\begin{subequations}
\label{SBoCo}
\begin{align}
\label{SBoCo1}
(1+u_{11})\,b_{-a}(x)-i(1-u_{11})\,c_{-a}(x)+u_{12}\,b_{a}(x)+iu_{12}\,c_{a}(x)&=0\,,\\
\label{SBoCo2}
u_{21}\,b_{-a}(x)+iu_{21}\,c_{-a}(x)+(1+u_{22})\,b_{a}(x)-i(1-u_{22})\,c_{a}(x)&=0
\end{align}
\end{subequations}
\begin{rem}
\label{CUM}%
Since the form \(\Omega_M(x,y)\) is \emph{anti}linear with respect
to the argument \(y\): \(\Omega_M(x,\mu{}y)=
\overline{\mu}\,\Omega_M(x,y)\) for \(\mu\in\mathbb{C}\), the
numbers \(i,-i\) which occurs in \eqref{VGeE} must be replaced
with the numbers \(-i,i\) in appropriate positions in the equality
\eqref{SBoCo}. For the same reason, the numbers \(u_{pq}\) which
occurs in \eqref{VGeE} must be replaced with the numbers
\(\overline{u_{pq}}\) in \eqref{SBoCo}. However to simplify the
notation, we replace the number \(u_{pq}\) with the number
\(u_{qp}\) rather with the numbers \(\overline{u_{pq}}\).
In other words,  we use the matrix \(U^{\ast}\)
as a matrix which parameterizes the set of all
\(\Omega_M\)-self-orthogonal subspaces.  The matrix \(U^\ast\) is
an \emph{arbitrary} unitary matrix if \(U\) is an arbitrary unitary matrix.
\end{rem}
\begin{defn}
\label{SAEMOp}%
Let \(U\) be an arbitrary \(2\times2\)  matrix. The operator \(\mathcal{M}_U\)
is
de\-fin\-ed  in the following way:
\begin{enumerate}
\item
The domain of definition \(\mathcal{D}_{\mathcal{M}_U}\) of the
operator \(\mathcal{M}_U\) is the set of all
\(x(t)\in\mathcal{D}_{\mathcal{M}_{\textup{max}}}\) which satisfy
the conditions \eqref{SBoCo1}-\eqref{SBoCo2}, \eqref{GBCo}.
\item
For \(x\in\mathcal{D}_{\mathcal{M}_U}\), the action of the operator \(\mathcal{M}_{U}\) is
\[\mathcal{M}_{U}x=\mathcal{M}_{\textup{max}}x.\]
\end{enumerate}
\end{defn}
\begin{rem}
\label{JDD}%
 In view of \eqref{DoDeMiOr} and \eqref{SBoCo}, for \emph{any}
matrix \(U\),
\[\mathcal{D}_{\mathcal{M}_{\textup{min}}}\subseteq\mathcal{D}_{\mathcal{M}_U}\,.\]
Thus for any matrix \(U\), the operator \(\mathcal{M}_U\) is an
extension of the operator \(\mathcal{U}_{\textup{min}}\):
\begin{equation}
\label{Betw}%
\mathcal{M}_{\textup{min}}\subseteq\mathcal{M}_U\subseteq
\mathcal{M}_{\textup{max}}\,.
\end{equation}
The equalities \eqref{SBoCo} which determine the domain of
definition of the extension \(\mathcal{M}_U\) can be considered as
\emph{boundary conditions} posed on functions
\(x\in\mathcal{D}_{\mathcal{M}_{\textup{max}}}\). \textup{(}See
\textup{Remark \ref{GeBV}.)}
\end{rem}
The following Theorem is a reformulation of Lemma \ref{ODSS} in
the language of extensions of operators.
\begin{thm}{\ } 
\label{DSAE}
\begin{enumerate}
\item
If \(U\) is an unitary matrix, then the operator \(\mathcal{M}_U\)
is a self-adjoint differential operator, and \ %
\(\mathcal{M}_{\textup{min}}\subset\mathcal{M}_U\subset\mathcal{M}_{\textup{max}}\)\,.
\item
Every differential operator \(\mathcal{M}\) which is self-adjoint
extension of the minimal differential operator
\(\mathcal{M}_{\textup{min}}\),
\(\mathcal{M}_{\textup{min}}\subset\mathcal{M}\subset\mathcal{M}_{\textup{max}}\),
is of the form
\(\mathcal{M}=\mathcal{M}_U\) for some unitary matrix \(U\).
\item
For unitary matrices \(U_1,\,U_2\),
\begin{equation*}
(U_1=U_2)\Leftrightarrow(\mathcal{M}_{U_1}=\mathcal{M}_{U_2})\,.
\end{equation*}
\end{enumerate}
\end{thm}

The equalities \eqref{LDO}, which relate the formal Legendre operator \(L\) and formal prolate spheroid operator \(M\), lead to the equalities
\begin{subequations}
\label{ERML}
\begin{gather}
\mathcal{L}_{\textup{max}}=\mathcal{M}_{\textup{max}}+Q,\label{ERML1}\\
\mathcal{L}_{\textup{min}}=\mathcal{M}_{\textup{min}}+Q,\label{ERML2}
\end{gather}
where \(Q\) is the multiplication operator:
\begin{equation}
\label{DMO}
\mathcal{D}_Q=L^2([-a,a]), \quad (Qx)(t)=t^2x(t).
\end{equation}
\end{subequations}
The operator \(Q\) is a bounded self-adjoint operator:
\begin{equation}
\label{DMOp}
Q=Q^{\ast}.
\end{equation}
So there are no problems with the equalities \eqref{ERML}. We may consider the operators
in the right hand sides of the equalities \eqref{ERML} as \emph{definitions} for the operators in the left hand sides. In particular, the domains of definition coincide:
\begin{equation}
\label{CoDD}
\mathcal{D}_{\mathcal{L}_{\textup{max}}}=\mathcal{D}_{\mathcal{M}_{\textup{max}}},\quad
\mathcal{D}_{\mathcal{L}_{\textup{min}}}=\mathcal{D}_{\mathcal{M}_{\textup{min}}}.
\end{equation}
The relations
\begin{equation*}%
\mathcal{L}_{{}_\textup{min}}\subseteq(\mathcal{L}_{{}_\textup{min}})^\ast\,;\quad
(\mathcal{L}_{{}_\textup{min}})^\ast=\mathcal{L}_{{}_\textup{max}},\quad
(\mathcal{L}_{{}_\textup{max}})^\ast=\mathcal{L}_{{}_\textup{min}}\,.
\end{equation*}%
are consequences of the relations \eqref{MinSym}, \eqref{ATDO}, of the definitions
\eqref{ERML} and of the equality \eqref{DMOp}. In view of \eqref{DMOp}, the boundary forms
\(\Omega_M\) and \(\Omega_L\) coincide. The boundary linear forms related to the
operators \(L\) and \(M\) are the same and are expressed by \eqref{GBCo}.
Finally the self-adjoint extensions of the symmetric operator \(\mathcal{L}_{{}_\textup{min}}\) are in one-to-one correspondence with \(2\times2\) unitary
matrices \(U\). This correspondence is of the form \(U\Leftrightarrow\mathcal{L}_U\), where
the domain of definitions \(\mathcal{D}_{\mathcal{L}_{U}}=\mathcal{D}_{\mathcal{M}_{U}}\)
is described by liner boundary conditions \eqref{SBoCo}. Moreover the equality
\begin{equation}
\label{Pot}
\mathcal{L}_U=\mathcal{M}_U+Q
\end{equation}
holds.
\section{Spectral analysis of the operators \(\mathcal{L}_U\).}
The matrix \(I\) is \(2\times2\) identity matrix:
\(I=\bigl[\begin{smallmatrix}1&0\\0&1\end{smallmatrix}\bigr]\).
The operators \(\mathcal{L}_I\) and \(\mathcal{M}_I\) are the
operators \(\mathcal{L}_U\) and \(\mathcal{M}_U\) corresponding to the choice \(U=I\). In particular, for \(U=I\) the boundary conditions \eqref{SBoCo} take the form
\begin{equation}
\label{FBC}
\lim\limits_{|\xi|\to{}a-0}\bigg(1-\frac{\xi^2}{a^2}\bigg)\frac{dx(\xi)}{d\xi}=0,
\quad\forall\,x\in\mathcal{D}_{\mathcal{M}_{I}}=\mathcal{D}_{\mathcal{L}_{I}},
\end{equation}
\begin{lem}
\label{CAEP}
Let \(x\in\mathcal{D}_{\mathcal{M}_{I}}\), and
\begin{equation}
\label{Pr}
\int\limits_{-a}^{a}\bigg|\frac{d\,\,}%
{d\xi}\bigg(\bigg(1-\frac{\xi^2}{a^2}\bigg)\frac{dx(\xi)}{d\xi}\bigg)\bigg|^2\,d\xi=C^2<\infty, \ \ C=C(x)>0.
\end{equation}
Then
\begin{equation}
\label{EPV}
\bigg|\frac{dx(t)}{dt}\bigg|\leq\sqrt{2}\,C\,a^{3/2},\quad \forall\,t\in(-a,a).
\end{equation}
\end{lem}
\begin{proof}
From \eqref{Pr} and the Schwarz inequality we obtain
\begin{equation}
\label{SI}
\int\limits_{-a}^{a}\bigg|\frac{d\,\,}%
{d\xi}\bigg(\bigg(1-\frac{\xi^2}{a^2}\bigg)\frac{dx(\xi)}{d\xi}\bigg)\bigg|\,d\xi\leq
\sqrt{2a}\,C.
\end{equation}
From \eqref{FBC} and \eqref{SI} we derive the inequality
\begin{equation*}
\bigg|\bigg(1-\frac{t^2}{a^2}\bigg)\frac{dx(t)}{dt}\bigg|\leq\sqrt{2a}\,C\,\min(a+t,a-t),
\quad\forall\,t\in(-a,a).
\end{equation*}
Since
\(\min(a+t,a-t)\leq\,a\big(1-\frac{t^2}{a^2}\big)\), we obtain the inequality
\eqref{EPV}.
\end{proof}
\begin{subequations}
\label{Cor}
\begin{lem}
\label{BNE}
Let \(x\in\mathcal{D}_{\mathcal{L}_{I}}\). Then the limits
\begin{equation}
\label{Cor1}
x(-a+0)=\lim\limits_{t\to{}-a+0}x(t),\quad x(a-0)=\lim\limits_{t\to{}a-0}x(t)
\end{equation}
 exist and are finite:
 \begin{equation}
\label{Cor2}
|x(-a+0)|<\infty,\quad |x(a-0)|<\infty.
 \end{equation}
\end{lem}
\end{subequations}
\newpage
\begin{thm}{\ }
\label{DiSp}%
\begin{enumerate}
\item
The self-adjoint operator \(\mathcal{M}_I\) is non-negative:
\begin{equation}
\label{PoM}
\langle\mathcal{M}_I\,x,x\rangle\geq0\,,\quad \forall\,x\in\mathcal{D}_{\mathcal{M}_{I}},\,x\not=0.
\end{equation}
\item
The self-adjoint operator \(\mathcal{L}_I\) is positive:
\begin{equation}
\label{PoL}
\langle\mathcal{L}_I\,x,x\rangle>0\,,\quad \forall\,x\in\mathcal{D}_{\mathcal{L}_{I}},\,x\not=0.
\end{equation}
\end{enumerate}
\end{thm}
\begin{proof}{\ }\\
\textbf{1.} Integrating by parts, we obtain
\[
\int\limits_{-a}^{a}-\frac{d\,\,}{dt}
\bigg(\Big(1-\frac{t^2}{a^2}\Big)\frac{dx}{dt}\bigg)\cdot\overline{x(t)}\,dt=
\int\limits_{-a}^{a}\Big(1-\frac{t^2}{a^2}\Big)\Big|\frac{dx}{dt}\Big|^2dt.
\]
In view of \eqref{FBC} and \eqref{Cor}, the summands corresponding to the endpoints
\(-a\) and \(a\) disappear. The last equality can be interpreted as
\[\langle\mathcal{M}_Ix,x\rangle=
\int\limits_{-a}^{a}\Big(1-\frac{t^2}{a^2}\Big)\Big|\frac{dx}{dt}\Big|^2dt,
\quad \forall\,x\in\mathcal{D}_{\mathcal{M}_{I}}.\]
So the inequality \eqref{PoM} holds.\\
\textbf{2}. The operator \(Q\) is positive:
\begin{equation}
\label{PoQ}
\langle{}Qx,x\rangle>0,\quad \forall\,x\in{}L^2([-a,a]),\,x\not=0.
\end{equation}
The inequality \eqref{PoL} is a consequence of the inequalities \eqref{PoM},\eqref{PoQ}
and of the equality \eqref{Pot} with \(U=I\).
\end{proof}
Let \(\mathcal{I}\) be the identity operator in \(L^2([-a,a])\).
\begin{lem}{\ }
\label{MDS}
Given \(\lambda\in\mathbb{C}\setminus[0,\infty)\), the operators \
\((\mathcal{M}_I-\lambda\mathcal{I})^{-1}\) and \((\mathcal{L}_I-\lambda\mathcal{I})^{-1}\) are compact operators.
\end{lem}
\begin{proof} Since both operators \(\mathcal{M}_I\) and \(\mathcal{L}_I\)
are self-adjoint and non-negative, both resolvents
\((\mathcal{M}_I-\lambda\mathcal{I})^{-1}\) and \((\mathcal{L}_I-\lambda\mathcal{I})^{-1}\)
exist and are bounded operators.

The spectral analysis of the operator \(\mathcal{M}_I\) can be done explicitly.
Let \(P_k(t)\) be the Legendre polynomials:
\[P_k(t)=\frac{1}{2^kk!}\frac{d^k\,\,}{dt^k}\,(t^2-1)^k,\quad k=0,1,2,\,\ldots\,,\]
and
\begin{equation}
 v_k(t)=P_k(t/a),\quad t\in[-a,a],\quad \quad k=0,1,2,\,\ldots\,.
\end{equation}
The system \(\{v_k(t)\}_{k=0,1,2,\,\ldots}\) is a complete orthogonal system in
\(L^2([-a,a])\). The functions \(v_k(t)\) are eigenfunctions of the operator \(\mathcal{M}_I\):
\begin{subequations}
\label{SAM}
\begin{align}
(\mathcal{M}_Iv_k)(t)&=\mu_kv_k(t),\label{SAM1}\\
\intertext{where}
\mu_k=\frac{k(k+1)}{a^2}\ccomma &\  \quad k=0,1,2,\,\ldots\,.\label{SAM2}
\end{align}
Thus the operator \(\mathcal{M}_I\) is an operator with discrete spectrum and the resolvent
\((\mathcal{M}_I-\lambda\mathcal{I})^{-1}\) is a compact operator.
\end{subequations}
Since
\[(\mathcal{L}_I-\lambda\mathcal{I})^{-1}=(\mathcal{M}_I-\lambda\mathcal{I})^{-1}-
(\mathcal{M}_I-\lambda\mathcal{I})^{-1}Q(\mathcal{L}_I-\lambda\mathcal{I})^{-1},\]
the operator \((\mathcal{L}_I-\lambda\mathcal{I})^{-1}\) is a compact operator as well.
\end{proof}
\begin{lem}{\ }
\label{MDSU}
Given \(\lambda\in\mathbb{C}\setminus(-\infty,\infty)\) and an unitary matrix \(U\), the operator \((\mathcal{L}_U-\lambda\mathcal{I})^{-1}\) is a compact operator.
\end{lem}
\begin{proof}
Since \(\lambda\notin\mathbb{R}\), both resolvents
\((\mathcal{L}_U-\lambda\mathcal{I})^{-1}\), \ \((\mathcal{L}_I-\lambda\mathcal{I})^{-1}\)
exist. Since both operators \(\mathcal{L}_U\) and \(\mathcal{L}_I\) and extensions of the same operator \(\mathcal{L}_{{}_\textup{min}}\) with deficiency indices
\begin{equation}
\label{DeIn}
n_+(\mathcal{L}_{{}_\textup{min}})=n_-(\mathcal{L}_{{}_\textup{min}})=2,
\end{equation}
 the difference
of the resolvents \((\mathcal{L}_U-\lambda\mathcal{I})^{-1}-(\mathcal{L}_I-\lambda\mathcal{I})^{-1}\)
is an operator which rank does not exceed two. According to Lemma \eqref{MDS}, the operator
\((\mathcal{L}_I-\lambda\mathcal{I})^{-1}\) is compact. Hence the operator
\((\mathcal{L}_U-\lambda\mathcal{I})^{-1}\) is compact.
\end{proof}

\begin{thm}
\label{DSOU}{\ }\\[-3.5ex]
\begin{enumerate}
\item
For any unitary matrix
\(U=\bigl[\begin{smallmatrix}u_{11}&u_{12}\\u_{21}&u_{22}\end{smallmatrix}\bigr]\),
the spectrum of the operator \(\mathcal{L}_U\) is discrete. This spectrum is formed
by the sequence \(\{\lambda_k(\mathcal{L}_U)\}_{1\leq k<\infty}\) of the eigenvalues of
\(\mathcal{L}_U\):\\[-3.5ex]
 \begin{multline}
 \label{EnCo}
 \hspace*{3.0ex}\lambda_1(\mathcal{L}_U)\leq\lambda_2(\mathcal{L}_U)\leq\,\ldots\,
\leq\lambda_k(\mathcal{L}_U)\leq\,\ldots\,, \\
 \lambda_k(\mathcal{L}_U)\to\infty\ \textup{as}\ k\to\infty.
\end{multline}
\item
 Not more than two of these eigenvalues can be negative:
 \begin{equation}
 \label{NeEi}
 \lambda_k(\mathcal{L}_U)\geq0,\quad 3\leq k<\infty.
 \end{equation}
 \item
 The multiplicity of the eigenvalue \(\lambda_k(\mathcal{L}_U)\) does not exceed two:
 \begin{equation}
 \label{Mul}
 \textup{mult}(\lambda_k(\mathcal{L}_U))\leq2,\quad 1\leq k<\infty.
 \end{equation}
 \item If at least one of the entries \(u_{11}\), \(u_{22}\) of the matrix \(U\) is equal to one:, i.e if \((1-u_{11})((1-u_{22})=1\),
 then all eigenvalues \(\lambda_k(\mathcal{L}_U)\) are of multiplicity one:
 \begin{equation}
 \label{Mulo}
 \textup{mult}(\lambda_k(\mathcal{L}_U))=1,\quad 1\leq k<\infty.
 \end{equation}
 \end{enumerate}

\end{thm}
\begin{proof} According to Lemma \ref{MDSU}, the spectrum of the self-adjoint operator \(\mathcal{L}_U\)  consists of isolated points which are eigenvalues. The operator
\(\mathcal{L}_U\) is an extension of the symmetric operator \(\mathcal{L}_{{}_\textup{min}}\)
which is non-negative. (The inequality \eqref{PoL} for \(x\in\mathcal{D}_{\mathcal{L}_{{}_\textup{min}}}\).) Since the deficiency indices
of the operator  \(\mathcal{L}_{{}_\textup{min}}\) are finite, \eqref{DeIn}, the spectrum
of the operator \(\mathcal{L}_U\) is bounded from below. Hence the sequence
\(\{\lambda_k(\mathcal{L}_U)\}_{1\leq k<\infty}\) of the eigenvalues of
\(\mathcal{L}_U\) can be enumerated such that the conditions \eqref{EnCo} holds. The
condition \eqref{NeEi} is a consequence of \cite[Theorem 18]{Kr}.  The inequality \eqref{Mul} holds because the equation \(\mathcal{L}_U\,x-\lambda x=0\) is a differential equation of order two. If \(u_{11}=1\) then the boundary condition
\eqref{SBoCo1} is of the form \(b_{-a}(x)=0\). According to a version of Lemma \ref{BNE},
formulated for the operator \(L\), any solution \(x(t,\lambda)\) of the eigenvalue problem \(\mathcal{L}_Ux=\lambda x\) is bounded as \(t\to-a+0\). According to
Lemma \ref{ABSNS}, any solution \(x(t,\lambda)\) of the differential equation \((Lx)(t,\lambda)=\lambda x(t,\lambda)\) must be of the form \eqref{LComM}. Since the function \(x_{1}^{-}(t,\lambda)\) is bounded and the function \(x_{2}^{-}(t,\lambda)\) is unbounded as \(t\to-a+0\), the coefficient \(c_2^{-}\) in \eqref{LComM} must vanish.
\end{proof}
Among all self-adjoint extensions \(\mathcal{L}_U\) of the minimal symmetric non-negative operator \(\mathcal{L}_{{}_\textup{min}}\), we distinguish the extension \(\mathcal{L}_I\)
which corresponds to the choice of the identity matrix \(I\) as the matrix \(U\).
The operator \(\mathcal{L}_I\) plays a special role. We shall see in the next section
that among all extensions \(\mathcal{L}_U\) of the operator \(\mathcal{L}_{{}_\textup{min}}\), only the operator \(\mathcal{L}_I\) commutes with
the truncated Fourier operator \(\mathscr{F}_{[-a,a]}\).

\begin{thm}
\label{DSOIT}{\ }\\[-2.5ex]
\begin{enumerate}
\item
 The spectrum of the operator \(\mathcal{L}_I\) is formed
by the sequence \(\{\lambda_k\}_{1\leq k<\infty}\) of positive eigenvalues of
multiplicity one:
 \begin{equation}
 \label{EnCoI}
 \hspace*{3.0ex}0<\lambda_1<\lambda_2<\,\ldots\,
<\lambda_k<\,\ldots\,, \quad
 \lambda_k\to\infty\ \textup{as}\ k\to\infty.
\end{equation}
\item
\begin{subequations}
The system of the eigenfunctions \(\{\chi_k\}_{1\leq k<\infty}\):

\label{EigfI}
\begin{gather}
 \label{EigfI1}
 (L\chi_k)(t)=\lambda_k\chi_k(t),\quad t\in(-a,a),\\
 \label{EigfI2}
 b_{-a}(\chi_k)=0,\quad b^{a}(\chi_k)=0,
 \end{gather}
 is a complete orthogonal system in \(L^2([-a,a])\).
 \end{subequations}
\end{enumerate}
\end{thm}
\begin{defn}
The functions \(\chi_k(t)\), which are the eigenfunction of the boundary value problem
\eqref{EigfI}, are said to be \emph{the prolate spheroidal wave functions.}
\end{defn}
\begin{rem}
\label{CoTD}
Traditionally the prolate spheroidal wave functions \(\chi_k\) are defined as those solutions
of the equation \eqref{EigfI1} which are bounded on \((-a,a)\):
\begin{equation}
\label{Alt}
\sup\limits_{t\in(-a,a)}|\chi_{k}(t)|<\infty
\end{equation}
The traditional definitions is equivalent to the definition \(\chi_k\) by means of the
eigenvalue problem \eqref{EigfI1}, \eqref{EigfI2}.
\end{rem}

\section{Commutator of the operators \(\mathscr{F}_E\) and \(\mathcal{L}_U\).}
For \(x\in\mathcal{D}_{\mathcal{L}_{\textup{max}}}\), let us calculate the difference
\(\mathscr{F}_E\mathcal{L}_{\textup{max}}x-\mathcal{L}_{\textup{max}}\mathscr{F}_Ex\).
Since \(\mathcal{L}_{\textup{max}}x\in{}L^2([-a,a])\), the expression
\(\mathscr{F}_E(\mathcal{L}_{\textup{max}}x)\) is defined.
The functions \((\mathscr{F}_Ex)(t)\) and \((\mathscr{F}_E\mathcal{L}_{\textup{max}}x)(t)\) are  smooth  on the closed interval
\([-a,a]\). (In fact these function are analytic in the whole real axis.)
All the more, \(\mathscr{F}_Ex\in\mathcal{D}_{\mathcal{L}_{\textup{max}}}\).
Thus for \(x\in\mathcal{D}_{\mathcal{L}_{\textup{max}}}\),
the difference \(\mathscr{F}_E\mathcal{L}_{\textup{max}}x-\mathcal{L}_{\textup{max}}\mathscr{F}_Ex\)
is well defined.

Assuming that  \(x\in\mathcal{D}_{\mathcal{L}_{\textup{max}}}\) and that
\(-a<\alpha<\beta<a\),  we integrate
by parts twice\,%
\begin{multline}
\label{ffp}%
 \int\limits_{\alpha}^{\beta}
\left(-\frac{d\,\,}{d\xi}\Bigg(\bigg(1-\frac{\xi^2}{a^2}\bigg)\right)
\frac{dx(\xi)}{d\xi}\Bigg)e^{it\xi}d\xi=\\[1.0ex]
=-\bigg(1-\frac{\xi^2}{a^2}\bigg)\frac{dx(\xi)}{d\xi}\,e^{it\xi}%
\bigg|_{\xi=\alpha}^{\xi=\beta}
+it\bigg(1-\frac{\xi^2}{a^2}\bigg)x(\xi)%
e^{it\xi}\bigg|_{\xi=\alpha}^{\xi=\beta}-\\[1.0ex]
-it
\int\limits_{\alpha}^{\beta}x(\xi)\,\frac{d\,\,}{d\xi}
\left(\bigg(1-\frac{\xi^2}{a^2}\bigg)e^{it\xi}\right)\,d\xi\,.
\end{multline}

For \(x\in\mathcal{D}_{\mathcal{L}_{\textup{max}}}\), both limits
\(\displaystyle\lim_{t\to\pm{}a}(1-t^2/a^2)\frac{dx(t)}{dt}\) exist, are finite,
and
\begin{subequations}
\label{EVT}
\begin{align}
\label{EVT1}
\lim_{t\to{}-a}(1-t^2/a^2)\frac{dx(t)}{dt}&=\ \frac{2}{a^2}\,b_{-a}(x)\,,\\
\label{EVT2}
\lim_{t\to+a}(1-t^2/a^2)\frac{dx(t)}{dt}&=-\frac{2}{a^2}\,\,b_{\,a}(x)\,.\,
\end{align}
\end{subequations}
where \(b_{-a}(x), b_{a}(x)\) are defined in \eqref{GBCo} and also appear
in the boundary conditions \eqref{SBoCo}. Since the limits in \eqref{EVT}
are finite, we conclude that \(|x(t)|=O(\ln(a^2-t^2))\) as \(t\to\pm{}a,\,|t|<a\).
All the more, for \(x\in\mathcal{D}_{\mathcal{L}_{\textup{max}}}\)
\begin{equation}
\label{ZDC}
\lim_{t\to-a+0}\bigg(1-\frac{t^2}{a^2}\bigg)x(t)=0\,.
\end{equation}
Passing to the limit in \eqref{ffp} and taking into account \eqref{ZDC} and \eqref{EVT},
we obtain
\begin{multline}
\label{ffpOm}%
 \int\limits_{-a}^{a}
\left(-\frac{d\,\,}{d\xi}\Bigg(\bigg(1-\frac{\xi^2}{a^2}\bigg)
\frac{dx(\xi)}{d\xi}\Bigg)\right)e^{it\xi}d\xi=
\frac{2}{a}\bigg(b_{a}(x)e^{iat}+b_{-a}(x)e^{-iat}\bigg)-
\\[1.0ex]
-it
\int\limits_{-a}^{a}x(\xi)\,\frac{d\,\,}{d\xi}
\left(\bigg(1-\frac{\xi^2}{a^2}\bigg)e^{it\xi}\right)\,d\xi\,.
\end{multline}
Transforming the last summand of the right hand side of \eqref{ffpOm},
we obtain
\begin{multline*}%
-it\int\limits_{-a}^{a}x(\xi)\,\frac{d\,\,}{d\xi}
\left(\bigg(1-\frac{\xi^2}{a^2}\bigg)e^{it\xi}\right)\,d\xi=\\[1.0ex]
=t^2\int\limits_{-a}^{a}x(\xi)e^{it\xi}\,d\xi+
\frac{it}{a^2}\int\limits_{-a}^{a}x(\xi)%
\frac{d\,\,}{d\xi}(\xi^2e^{it\xi})\,d\xi=
\end{multline*}%
\vspace{-2.0ex}
\hspace*{8.0ex}\bigg(\,\,since \(\dfrac{d\,\,}{d\xi}(\xi^2e^{it\xi})=\dfrac{d\,\,}{d\xi}%
\Big(-\dfrac{d^2\,\,}{dt^2}e^{it\xi}\Big)
=-\dfrac{d^2\,\,}{dt^2}\big(ite^{it\xi}\big)\)\,\,\bigg)\hspace*{2.0ex}
\vspace{2.0ex}
\begin{multline*}%
=t^2\int\limits_{-a}^{a}x(\xi)e^{it\xi}\,d\xi+\frac{t}{a^2}\,%
\frac{d^2\,\,}{dt^2}\bigg(t\int\limits_{-a}^{a}x(\xi)e^{it\xi}\,d\xi\bigg)=\\[1.0ex]
=t^2\int\limits_{-a}^{a}x(\xi)e^{it\xi}\,d\xi+\frac{d\,\,}{dt}%
\bigg(\frac{t^2}{a^2}\frac{d\,\,}{dt}\int\limits_{-a}^{a}x(\xi)e^{it\xi}\,d\xi\bigg)=
\end{multline*}
\begin{equation}
\label{Cont}
=t^2\int\limits_{-a}^{a}x(\xi)e^{it\xi}\,d\xi-
\frac{d\,\,}{dt}\Bigg(\bigg(1-\frac{t^2}{a^2}\bigg)\frac{d\,\,}{dt}%
\int\limits_{-a}^{a}x(\xi)e^{it\xi}\,d\xi\Bigg)-
\int\limits_{-a}^{a}\xi^2x(\xi)\,e^{it\xi}\,d\xi\,.
\end{equation}%
Unifying \eqref{ffpOm} and \eqref{Cont}, we obtain the equality
\begin{multline}%
\int\limits_{-a}^{a}
\left(\bigg(-\frac{d\,\,}{d\xi}\bigg(1-\frac{\xi^2}{a^2}\bigg)
\frac{d\,\,}{d\xi}+\xi^2\bigg)\,x(\xi)\right)e^{it\xi}\,d\xi=\\[1.0ex]
=\frac{2}{a}\bigg(b_{a}(x)e^{iat}+b_{-a}(x)e^{-iat}\bigg)
+\Bigg(-\frac{d\,\,}{dt}\bigg(1-\frac{t^2}{a^2}\bigg)\frac{d\,\,}{dt}+t^2\Bigg)%
\int\limits_{-a}^{a}x(\xi)e^{it\xi}\,d\xi\,.
\end{multline}%
We summarize the above calculation as
\begin{lem}
Let \(\mathscr{F}_E\) be the Fourier operator truncated on the
finite symmetric interval \(E=[-a,a]\). Let \(\mathcal{L}_{\textup{max}}\) be
the maximal differential operator with domain of definition %
\(\mathcal{D}_{\mathcal{L}_{\textup{max}}}\)
generated by the formal differential operator
\(\displaystyle{}L=-\frac{d\,\,}{dt}\bigg(1-\frac{t^2}{a^2}\bigg)\frac{d\,\,}{dt}+t^2\).
\textup{(}See \textup{Definition \ref{MaxDO}.)}

If  \(x\in\mathcal{D}_{\mathcal{L}_{\textup{max}}}\),  then  \(\mathscr{F}_Ex\in\mathcal{D}_{\mathcal{L}_{\textup{max}}}\),
and the equality holds
\begin{equation}
\label{ComRel}
(\mathscr{F}_E\mathcal{L}_{\textup{max}}x)(t)-
(\mathcal{L}_{\textup{max}}\mathscr{F}_Ex)(t)=
\frac{2}{a}\bigg(b_{a}(x)e^{iat}+b_{-a}(x)e^{-iat}\bigg)\,.
\end{equation}
\end{lem}
Every self-adjoint differential operator generated by the formal differential
operator \(L\) is a restriction of the maximal differential operator \(\mathcal{L}_{\textup{max}}\)
on the appropriate domain of definition. According to Theorem~\ref{DSAE},
the set of such self-adjoint operators coincides with the set of operators \(\mathcal{L}_U\),
where \(U\) is an arbitrary \(2\times2\) unitary matrix. The domain of definition
\(\mathcal{D}_{\mathcal{L}_U}\) of the operator \(\mathcal{L}_U\) is distinguished from the
domain \(\mathcal{D}_{\mathcal{L}_{\textup{max}}}\) by the boundary conditions \eqref{SBoCo}
constructed from \(U\). The next theorem answers the question which operators \(\mathcal{L}_U\)
commute with the truncated Fourier operator \(\mathscr{F}_E,\,E=[-a,a]\).
\begin{thm}{\ }\\[-1.0ex]
\label{WECF}
\begin{enumerate}
\item
If \, \(U=I\), where \(I\) is \(2\times2\)  identity matrix, then
the differential operator\,%
\footnotemark\,%
\(\mathcal{L}_I\) commutes with the truncated  %
 Fourier operator~%
\(\mathscr{F}_{[-a,a]}\):
\begin{equation}
\label{DCoRe}
\mathscr{F}_{[-a,a]}\mathcal{L}_I\,x=\mathcal{L}_I\mathscr{F}_{[-a,a]}\,x\quad \forall\, x\in%
\mathcal{D}_{\mathcal{L}_I}\,.
\end{equation}
\item
If \,\,\(U\not=I\), then the operator \(\mathcal{L}_U\) do not commute with the operator
\(\mathscr{F}_{[-a,a]}\):
\begin{enumerate}
\item
There exist vectors \(x\in\mathcal{D}_{\mathcal{L}_U}\) such that
\(\mathscr{F}_{[-a,a]}\in\mathcal{D}_{\mathcal{L}_U}\), so both operators
\(\mathscr{F}_{[-a,a]}\mathcal{L}_{U}\) and \(\mathcal{L}_{U}\mathscr{F}_{[-a,a]}\)
are applicable to \(x\), but
\begin{equation}
\mathscr{F}_{[-a,a]}\mathcal{L}_{U}x\not=
\mathcal{L}_{U}\mathscr{F}_{[-a,a]}x\,;
\end{equation}
\item
 There exist vectors \(x\in\mathcal{D}_{\mathcal{L}_U}\) such that
\(\mathscr{F}_{[-a,a]}x\not\in\mathcal{D}_{\mathcal{L}_U}\), so the operator
\(\mathcal{L}_{U}\mathscr{F}_{[-a,a]}\) even can not be applied to such \(x\).
\end{enumerate}
 \footnotetext{\,\(\mathcal{L}_I=\mathcal{L}_U \textup{ for
}U=I\).} \addtocounter{footnote}{1}
\end{enumerate}
\end{thm}
\begin{proof}{\ }\\
\textbf{1}.
For \(U=I\), the boundary conditions \eqref{SBoCo} take the form
\begin{equation}
\label{DBoCo}
b_{-a}(x)=0,\quad b_{a}(x)=0\,.
\end{equation}
Thus, the domain of definition \(\mathcal{D}_{\mathcal{L}_I}\) of the operator
\(\mathcal{L}_I\) is:
\begin{equation}
\label{DDLI}
\mathcal{D}_{\mathcal{L}_I}=\big\lbrace{}x:\,x\in\mathcal{D}_{\mathcal{L}_{\textup{max}}},\
b_{-a}(x)=0,\,b_{a}(x)=0\big\rbrace\,.
\end{equation}
Every function \(x(t)\) on \((-a,a)\) which derivative is
bounded:
\(\sup\limits_{t\in(-a,a)}|x^\prime(t)|<\infty\), belongs to
\(\mathcal{D}_{\mathcal{L}_{\textup{max}}}\). Moreover, according
to \eqref{GBCo}, every such a function  satisfies the boundary
condition \eqref{DDLI}, i.e. \(b_{-a}(x)=0\),\,\(b_{a}(x)=0\).
Hence \emph{every smooth function on \((-a,a)\) which derivative
is bounded on \((-a,a)\), belongs to domain of definition
\(\mathcal{D}_{\mathcal{L}_I}\) of the operator
\(\mathcal{L}_I\)}. In particular, if \(x\in{}L^2((-a,a)\) and
\(y=\mathscr{F}_{[-a,a]}x\), then
\(y\in\mathcal{D}_{\mathcal{L}_I}\).
Thus for \(x\in\mathcal{D}_{\mathcal{L}_I}\) both summands in the expression %
\(\mathscr{F}_{[-a,a]}\mathcal{L}_Ix-\mathcal{L}_I\mathscr{F}_{[-a,a]}x\)
are well defined. Since the operator \(\mathcal{L}_I\) is a restriction of the
operator \(\mathcal{L}_{\textup{max}}\), then
\begin{equation*}%
\mathscr{F}_{[-a,a]}\mathcal{L}_Ix-\mathcal{L}_I\mathscr{F}_{[-a,a]}x=
\mathscr{F}_{[-a,a]}\mathcal{L}_{\textup{max}}x-%
\mathcal{L}_{\textup{max}}\mathscr{F}_{[-a,a]}x\,\,\textup{ for }
x\in\mathcal{D}_{\mathcal{L}_I}\,.
\end{equation*}
In view of \eqref{ComRel} and \eqref{DBoCo}, the equality \eqref{DCoRe} holds.\\[1.0ex]
\textbf{2}. Let \(U\not=I\). Then at least of one value \(u_{11}-1\) or \(u_{22}-1\)
differs from zero. For definiteness, let \ \(u_{11}-1\not=0\). Set
\begin{equation}%
\label{SpX}%
\gamma=\frac{1+u_{11}}{i(1-u_{11})}\,,\quad x(t)=\psi_{-}(t)+\gamma\varphi_{-}(t)+x_0(t),
\end{equation}
where \(x_0(t)\) is a smooth function which support is a compact subset of the
\emph{open} interval \((-a,a)\):
\begin{equation}%
\label{SuppX}%
\textup{supp}\,x_0\Subset(-a,a)\,.
\end{equation}
The function \(x_0\) will be chosen later. According to \eqref{Rewr}, \eqref{SuppX}
and the choice of \(\gamma\), \emph{for any choice of} \(x_0(t)\),
the function \(x(t)\) from \eqref{SpX} satisfy the
boundary conditions \eqref{SBoCo}.
Thus,
\begin{equation}
x(t)\in\mathcal{D}_{\mathcal{L}_U}\,.
\end{equation}
for any choice of \(x_0\). Moreover
\begin{equation}
\label{but}
b_{-a}(x)=1,\quad  b_{a}(x)=0\,.
\end{equation}
For the function \(y(t)=(\mathscr{F}_{(-a,a)}x)(t)\),
the boundary conditions \eqref{SBoCo} either hold, or does not hold.
This depends on the choice of the function \(x_0\). If \eqref{SBoCo}
hold for this \(y\), then \(\mathscr{F}_{(-a,a)}x\in\mathcal{D}_{\mathcal{L}_{U}}\)
and the equality
\eqref{ComRel} can be interpreted as the equality
\begin{equation}
\label{Exam1}
(\mathscr{F}_{(-a,a)}\mathcal{L}_{U}x)(t)-
(\mathcal{L}_{U}\mathscr{F}_{(-a,a)}x)(t)=
\frac{2}{a}\bigg(b_{a}(x)e^{iat}+b_{-a}(x)e^{-iat}\bigg)\,.
\end{equation}
In view of \eqref{but}, \((\mathscr{F}_{(-a,a)}\mathcal{L}_{U}x)(t)-
(\mathcal{L}_{U}\mathscr{F}_{(-a,a)}x)(t)\not=0\).

Let us show that both of the possibilities %
\(\mathscr{F}_{(-a,a)}x\in\mathcal{D}_{\mathcal{L}_{U}}\) and
\(\mathscr{F}_{(-a,a)}x\not\in\mathcal{D}_{\mathcal{L}_{U}}\)
are realizable. Since the function \(y(t)=(\mathscr{F}_{(-a,a)}x)(t)\)
is smooth on \([-a,a]\),
\begin{equation*}%
b_{-a}(y)=0,\,b_{a}(y)=0,\,\,c_{-a}(y)=-y(-a),\,c_{-a}(y)=-y(a)\,.
\end{equation*}%
Thus the boundary conditions \eqref{SBoCo}
take the form
\begin{subequations}
\label{SpBC}
\begin{gather}
\label{SpBC1}
(1-u_{11})y(-a)-u_{12}y(a)=0\,,\\
\label{SpBC2}
u_{21}y(-a)-(1-u_{22})y(a)=0\,.
\end{gather}
\end{subequations}
If, using the freedom of choice of the function \(x_0(t)\) in \eqref{SpX}, we can arbitrary
prescribe the values \(y(-a)\) and \(y(a)\), then we can either satisfy the
boundary conditions \eqref{SpBC} (prescribing \(y(-a)=0,\,y(a)=0\)),
or violate them (if \(u_{11}\not=1\),
we prescribe \(y(-a)=1,\,y(a)=0\), if  \(u_{22}\not=1\),
we prescribe \(y(-a)~=~0,\,y(a)~=~1\).) The reference to Lemma below
finishes the proof.
\end{proof}
\begin{lem}
Given the complex numbers \(y_1\) and \(y_2\), there exists a smooth function
\(x_0(t)\) on \([-a,a]\) which possesses the properties:
\begin{enumerate}
\item \(\hspace*{20.0ex}\textup{supp}\,x_0\Subset(-a,a)\,.\)
\item \(\hspace*{10.0ex}y_0(-a)=y_1,\ y_0(a)=y_2\), where \(y_0=\mathscr{F}_{[-a,a]}(x_0)\).
\end{enumerate}
\end{lem}

\noindent


\end{document}